\documentclass[11pt]{article}
\usepackage[pagewise]{lineno} 
\usepackage{amsmath,amsfonts,amssymb,amsthm}
\usepackage{mathtools,mathrsfs,yhmath}
\usepackage{graphicx,color,xcolor}
\usepackage{enumitem}
\usepackage{blindtext}
\usepackage{multicol}
\usepackage{color}
\usepackage{comment}
\usepackage{wrapfig}
\usepackage{dsfont}
\usepackage{graphicx}
\usepackage{thmtools}
\usepackage[backend=biber, sorting=nty, doi=false, url=false, isbn=false, maxbibnames=5]{biblatex}
\bibliography{SEP}

\newcommand\blfootnote[1]{%
	\begingroup
	\renewcommand\thefootnote{}\footnote{#1}%
	\addtocounter{footnote}{-1}%
	\endgroup
}

 \allowdisplaybreaks
\numberwithin{equation}{section}
\numberwithin{equation}{section}
\newtheorem{theorem}{Theorem}[section]
\newtheorem{lemma}[theorem]{Lemma}
\newtheorem{proposition}[theorem]{Proposition}

\newtheorem{definition}[theorem]{Definition}

\newtheorem{remark}[theorem]{Remark}
\newcommand\numberthis{\addtocounter{equation}{1}\tag{\theequation}}
\usepackage{hyperref}
\hypersetup{colorlinks=true, linkcolor=blue, filecolor=magenta, urlcolor=blue}

\usepackage[capitalize, nameinlink]{cleveref}
\crefname{enumi}{}{parts}
\crefname{assumption}{Assumption}{Assumptions}

\newcommand{\Vol}{\mathrm{Vol}}
\renewcommand{\P}{\mathbb{P}}
\newcommand{\E}{\mathbb{E}}

\newcommand{\norm}[1]{\left\lVert#1\right\rVert}
\newcommand{\Var}{\mathrm{Var}}
\newcommand{\abs}[1]{\left\lvert#1\right\rvert}
\newcommand{\Supp}{\mathrm{Supp}}
\renewcommand{\L}{\mathcal{L}}
\newcommand{\R}{\mathbb{R}}

\renewcommand{\H}{\mathcal{H}}
\newcommand{\V}{\mathcal{V}}

\newcommand{\vh}{(\Supp\,\varphi)_{h_1}}
\newcommand{\vhh}{(\Supp\,\varphi)_{2h_1}}

\DeclareMathOperator{\divergence}{div}
\newcommand{\M}{M}
\DeclareMathOperator{\Ric}{Ric}

\DeclareMathOperator{\Hess}
{Hess}

\DeclareMathOperator{\Tr}{\mathrm{Tr}}

\DeclareMathOperator{\Ind}{\mathbf{1}}

\newcommand{\I}{\mathcal{I}}
\DeclareMathOperator{\RW}{(RW)}

\DeclareMathOperator{\URW}{(URW)}
\DeclareMathOperator{\SEP}{(SEP)}

\makeatletter
\newcommand{\biggg}{\bBigg@\thr@@}
\newcommand{\Biggg}{\bBigg@{3.5}}

\makeatother

\hypersetup{pdftitle={SEP}}
\hypersetup{pdfauthor={Jonathan Junn\'e, Frank Redig \& Rik Versendaal}}

\author{Jonathan Junn\'e\footnote{Delft Institute of Applied Mathematics, Faculty of Electrical Engineering, Mathematics and Computer Science, Delft University of Technology, Mekelweg 4, 2628CD Delft, The Netherlands. \href{mailto:j.junne@tudelft.nl}{J.Junne@tudelft.nl} \& \href{mailto:F.H.J.Redig@tudelft.nl}{F.H.J.Redig@tudelft.nl} \&
\href{mailto:R.Versendaal@tudelft.nl}
{R.Versendaal@tudelft.nl}}
\and Frank Redig$^{\ast}$
\and Rik Versendaal$^{\ast}$}

\title{Hydrodynamic Limit of the Symmetric Exclusion Process on Complete Riemannian Manifolds and Principal Bundles}

\begin{document}
\maketitle

\begin{abstract}
We prove that the hydrodynamic limit of the symmetric exclusion process (SEP) is a Fokker-Planck equation in the setting of Poisson random neighborhood graphs approximating a weighted Riemannian manifold with Ricci curvature bounded from below. We also consider the lift of the SEP to a principal bundle, and obtain a Fokker-Planck equation with a weighted horizontal Laplacian as its hydrodynamic limit. 
Both results significantly extend the geometric settings in which one can prove the hydrodynamic limit from duality combined with convergence of the single particle random walk towards a diffusion process.
\end{abstract}

\section{Introduction}
\label{Sec:Introduction}
Interacting particle systems serve as microscopic models from which macroscopic non-equilibrium behaviour can be derived. An example is the hydrodynamic limit where one derives partial differential equations for  macroscopic quantities such as the particle density by appropriate scaling of space and time in the microscopic particle dynamics,
see e.g. the monographs of Kipnis and Landim \cite{kipnis_scaling_1999}, DeMasi and Presutti \cite{demasi_mathematical_1992}, and Spohn \cite{spohn_large_1991}. A well-known and thoroughly studied example of an interacting system is the symmetric exclusion process (SEP) introduced by Spitzer \cite{spitzer_interaction_1970} (see also Liggett \cite{liggett_interacting_2005}, Chapter 8), where particles perform random walks with the restriction of having at most one particle per site. For this model in the standard setting of the $d$-dimensional lattice with nearest neighbour jumps, the hydrodynamic limit is the heat equation which can be derived after diffusive rescaling of space and time.

At present, hydrodynamic limits have been considered mostly on the lattice $\mathbb{Z}^d$ or on a discrete torus, and the corresponding macroscopic space is then either $\R^d$ or the $d$-dimensional torus. Let us mention a few cases where the hydrodynamic limit of the (SEP) was obtained beyond the Euclidean setting:
Faggionato \cite{Faggionato2022} considered symmteric exclusion processes in random environments, including Delaunay's triangulations starting from a homogeneous Poisson point process for which the limiting heat equation has a homogenized diffusion matrix. Chen and Gonçalvez \cite{1705.10290, chen1} studied hydrodynamic limits for particle systems on graphs approximating a fractal such as the Sierpinski gasket.

In \cite{van_ginkel_hydrodynamic_2020}, Redig and Van Ginkel started investigating the setting where the macroscopic space is a compact Riemannian manifold, with the aim of deriving macroscopic equations from interacting particle systems on discrete graphs approximating the manifold. More precisely, for a general class of neighborhood graphs that approximate a compact Riemannian manifold (based on uniformly chosen points according to the normalized volume measure), in \cite{van_ginkel_hydrodynamic_2020} the corresponding SEP is studied and a proof of the hydrodynamic limit  (yielding the heat equation) is given. 

In this paper, we extend the realm of these results in several ways. First, we drop the assumption of the compactness of the manifold. Second, we consider random neighborhood graphs whose vertices are given by a Poisson point process with Gibbs intensity. Third, whereas \cite{van_ginkel_hydrodynamic_2020} relied on the Wasserstein convergence of empirical distribution of i.i.d. points sampled from the normalised volume measure, here we rely on a simpler argument based on Bernstein's inequality to prove convergence of the rescaled random walks on the random neighbourhood graphs obtained from the Poisson point processes towards a diffusion process. Fourth, we lift the (SEP) to principal bundles, thus obtaining a simple sub-Riemannian setting where convergence to the heat equation corresponding to the horizontal Laplacian is obtained. As a result, starting from the corresponding (SEP), we obtain several new variants of the heat equation, including the heat equation with weighted Laplacian 
$e^U\divergence(e^{-U} \nabla)$, where $U$ is the potential of the Gibbs measure, replacing the Laplace-Beltrami operator. In the principal bundle setting, we obtain a horizontal heat equation with weighted horizontal Laplacian which can be considered as a first result in a sub-Riemannian setting. These results extend significantly the hydrodynamic equations that one can obtain from the (SEP), both in extending the geometric setting, but even in the Euclidean setting, where the heat equation with a weighted Laplacian is new.

The (SEP) is special because it manifests a microscopic closure of the expected density field; i.e., one does not need replacement lemmas to derive a macrosopic equation. This property, which can be viewed as a manifestation of self-duality, is crucial in the derivation of the (expected) macroscopic equation, which reduces to the proof
of the scaling limit of a single random walk (RW) to a diffusion process; the Kolmogorov forward equation of this diffusion process is then the hydrodynamic limit. Therefore, an important aspect is to consider suitable symmetric transition rates for (RW) on neighbourhood graphs in order to approximate a diffusion process with (horizontal) weighted Laplacian as generator. These transition rates are inspired from Hein, Audibert and von Luxburg
\cite{belkin_discrete_2008} and Gao \cite{gao_diffusion_2021}; in the context of horizontal diffusions, this requires a double scaling in the base space and the fibers. To prove weak convergence in path space of the trajectory of the empirical density, one needs to additionally control the quadratic variation of the corresponding Dynkin martingale. 

The rest of the paper is organized as follows: In \cref{sec:Main result} we introduce the (SEP) in our geometric setting and state the main results.
In \cref{sec:Existence} we establish existence of the (SEP) in our setting. For the case of non-compact manifolds, the approximating neighborhood graph is infinite, and the classical conditions of existence of SEP (from \cite{liggett_interacting_2005}) are not satisfied. However, via the graphical construction the proof of existence can be reduced to the proof of existence (i.e., non-explosion) of the random walk almost surely in the realization of the Poisson point process.  

In \cref{Sec:Abstract hydrodynamic limit via duality} we prove an abstract hydrodynamic limit for the (SEP) on complete manifolds which encompasses our setting. In \cref{Sec:Convergence of discrete random walks} we prove the almost sure convergence of a single (RW) on an increasing sequence of random neighbourhood graphs drawn from a (PPP) with Gibbs intensity. In \cref{bundlesect} we extend the setting to principal bundles by a horizontal lift procedure and obtain the corresponding hydrodynamic limit.

\section{Main result}\label{sec:Main result}
We consider a geodesically complete connected $m$-dimensional Riemannian manifold $M$ equipped with a Gibbs reference measure 
\begin{equation}\label{refmeas}
d\mu  \coloneqq e^{-U}d\Vol_M.
\end{equation} 
Here $\Vol_M$ denotes the Riemannian volume measure. We make the following assumptions on the manifold and the Gibbs measure:
\begin{itemize}
    \item[(A1)]\label{M} The manifold has a Ricci curvature lower bound; $\Ric_M \ge (m-1)\kappa$ for some $\kappa \in \R$.
    \item [(A2)]\label{U} The smooth potential $U \in C^\infty(M)$ of the Gibbs measure is such that $e^{-U}$ is bounded and the weighted manifold $(M, d\mu)$ is \emph{stochastically complete}; that is, for all $x \in M$,
    \begin{equation*}
        \int_M p_t(x, y) \: d\mu(y) = 1,
    \end{equation*}
    where $p_t$ denotes the heat kernel of the operator $\Delta - \nabla U \cdot \nabla$.
\end{itemize}
In the next definition we introduce a sequence of random neighborhood graphs $(G_N)_{N \ge 1}$ which approximate the manifold with corresponding measure $d\mu$. 
We denote by $d_M$ the Riemannian distance on $M$.
\begin{definition}[Random neighborhood graph (RNG)]\label{def:Random neighbourhood graph}
    Let $\Lambda^N$ be a \emph{Poisson point process ($\mathrm{PPP}$)} of intensity $Nd\mu$. 
    The \emph{vertices} $V_N$ consist of the points $X^i$ of $\Lambda^N$, and an \emph{undirected edge} is present between two vertices whenever they are close; i.e., $d_M(X^i, X^j) \le h_N$ for some \emph{bandwidth parameter} $h_N$. 
    The sequence of \emph{random neighborhood graphs (RNG)} is given by $(G_N \coloneqq (V_N, E_N))_{N \ge 1}$, where $E_N$ denotes the set of undirected edges.
\end{definition}
The symmetric exclusion process (SEP) we consider evolves on those random neighborhood graphs. 
Each site can accommodate at most one particle; this is the exclusion rule. 
Particles hop with symmetric edge weights over edges; jumps leading to more than one particle per site are suppressed.
\begin{definition}[Symmetric exclusion process (SEP)]\label{def:Symmetric exclusion process (SEP)}
    Let $G = (V, E)$ be a graph with symmetric weights $W$ on the edges. We denote a \emph{configuration of particles} by $\eta \in \Omega_V = \{0,1\}^V$ and the exchange of occupancies between vertices $x$ and $y$ by
    \begin{equation*}
        \eta^{x,y}(z) \coloneqq     \begin{cases}
            \eta(z) &\text{if }z \ne x, y, \\
            \eta(y) &\text{if }z = x, \\
            \eta(x) &\text{if }z = y.
        \end{cases}
    \end{equation*}
    The \emph{symmetric exclusion process (SEP)} is the continuous-time Markov jump process on the \emph{configuration space} $\Omega_V$ whose generator evaluated on functions depending only on a finite number of occupancies, namely, \emph{local functions} $f : \Omega_V \to \R$, is given by
    \begin{equation*}
        L_G^{\SEP}f(\eta) \coloneqq \sum_{(x, y) \in E} W(x, y) \left(f(\eta^{x,y}) - f(\eta)\right).
    \end{equation*}
\end{definition}
In our setting the
choice of symmetric weights take into account the geometry of the manifold $M$, and the choice of the potential $U$. More precisely, 
let $k : \R^+ \to \R^+$ be a bounded measurable function with compact support $[0, 1]$; we refer to $k$ as a \emph{kernel}.
We consider the following \emph{symmetric edge weights},
\begin{equation}\label{eq:symmetric weights for thm}
    W_N(X^i, X^j) \coloneqq W_N^{ij} \coloneqq \frac{h_N^{-m}k\left(\frac{d_M(X^i,X^j)}{h_N}\right)}{\sqrt{e^{-U(X^i)}e^{-U(X^j)}}}.
\end{equation}
\begin{remark}
The choice of the edge weights is motivated by the fact that the corresponding diffusively rescaled random walk with jumps over the edge $(xy)$ with rate 
$h_N^{-2} W_N(x,y)$ converges to a diffusion process with generator
$\Delta - \nabla U \cdot \nabla$. Because the edge weights are symmetric, by self-duality of the (SEP), the hydrodynamic limit of the associated (SEP) reduces to the scaling limit of a single random walk (see \cref{Sec:Abstract hydrodynamic limit via duality} for details).
\end{remark}
In order to formulate the result of the hydrodynamic limit, we introduce the empirical measure associated to a configuration of (SEP). 
More precisely, for a configuration $\eta^N$ of the (SEP) on the weighted graphs $G^N$ we define the macroscopic quantity
\begin{equation}\label{empmeas}
\widehat{\pi}^N=\frac{1}{N} \sum_{X^i \in V_N} \eta^N(X^i) \delta_{X^i}.
\end{equation}
The hydrodynamic limit can then be described in words as follows: if the empirical measure at time $0$ converges in an appropriate sense to $\rho_0 d\mu$, then after diffusive rescaling of time by $h_N^{-2}$, at time $t$, the
empirical measure converges to $\rho_t d\mu$ where $\rho_t$ satisfies \eqref{eq:heat drift} below.
We give below the precise formulation of this result on the hydrodynamic limit.

\begin{theorem}[Quenched hydrodynamic limit of the SEP]\label{thm:Hydrodynamic limit}
    Let $(G_N=(V_N,E_N))_{N \ge 1}$ be a sequence of $\mathrm{RNG}$ on $M$ in the sense of \cref{def:Random neighbourhood graph} with symmetric weights $(W_N)_{N\ge1}$ given by  \eqref{eq:symmetric weights for thm}. Assume that $(M, d\mu)$ satisfies $\mathrm{(A1)}$ and $\mathrm{(A2)}$. Let $\eta_0^N$ be the initial configuration of the $\mathrm{SEP}$ on $G_N$ with associated initial empirical measure $\widehat{\pi_0}^N$ that converges vaguely to $\rho_0 \,d\mu$ in probability; i.e.,
    \begin{equation*}
        \forall \varphi \in C_0(M), \quad \left\langle \varphi, d\widehat{\pi_0}^N \right\rangle \overset{\mathrm{pr.}}{\to} \left\langle \varphi, \rho_0 \: d\mu\right\rangle \quad \text{as }N\to +\infty, \quad \widehat{\pi_0}^N \coloneqq \frac{1}{N} \sum_{X^i \in V_N} \eta_0^N(X^i) \delta_{X^i}.
    \end{equation*}
    Suppose that the bandwidth parameters $(h_N)_{N\ge1}$ are chosen in such a way that both $h_N \to 0$ and $Nh_N^{m+2}/\log N \to +\infty$ as $N \to +\infty$.  Then for almost all realizations of the sequence of graphs, the trajectory of the empirical measure $t\mapsto \widehat{\pi_t}^N$ associated to the configuration $\eta_t^N$ that evolves according to the diffusively rescaled $\mathrm{SEP}$ with generator $h_N^{-2} L_{G_N}^{\SEP}$ converges in the Skorokhod topology to a deterministic trajectory $t\mapsto \rho_t\, d\mu$ whose density $\rho_t$ is the weak solution of
    \begin{equation}\label{eq:heat drift}
        \partial_t \rho_t = C_k\left(\Delta - \nabla U \cdot \nabla\right)\rho_t,
    \end{equation}
    with initial condition $\rho_0$ and for some constant $C_k$ that depends only on moments of the kernel $k$ in \eqref{eq:symmetric weights for thm}, provided there exists a unique weak solution.
\end{theorem}
\cref{thm:Hydrodynamic limit} is a direct consequence of an abstract hydrodynamic limit (\cref{thm:Abstract hydrodynamic limit}) and an almost sure pointwise consistency of a single random walk on the RNG (\cref{thm:almost sure convergence empirical URW to Delta 2(1-alpha) general}) that we establish in \cref{Sec:Abstract hydrodynamic limit via duality} and \cref{Sec:Convergence of discrete random walks} respectively. The existence of the SEP is the content of \cref{sec:Existence}.
\begin{remark}
    \hfill
    \begin{enumerate}
        \item In \cref{thm:Hydrodynamic limit}, we prepare the initial configuration to approximate $\rho_0 d\mu$ and consider the evolution of the trajectory $t \mapsto \rho_t\, d\mu$ for a density $\rho_t$ with respect to the reference Gibbs measure $d\mu$. Since the limiting operator $\Delta - \nabla U \cdot \nabla = e^U\nabla \cdot (e^{-U} \nabla)$ of a single random walk is reversible in $L^2(M, d\mu)$, the hydrodynamic limit \eqref{eq:heat drift} is indeed a Forward Kolmogorov equation. Regarding the existence of a unique weak solution, it is sufficient to require that $\Hess U \ge \beta$ for some $\beta \in \R$ so the weighted manifold $(M, d\mu)$ is stochastically complete (see e.g. Grigor'yan \cite{jorgenson_heat_2006}).
        \item If instead, one prepares the initial configuration such that it approximates $\rho_0 d\Vol_M$, then it is natural to look for the evolution of the density $\rho_t d\Vol_M$; that is, with respect to the volume measure. Because
        the $L^2(M, d\Vol_M)$-adjoint of $\Delta - \nabla U \cdot \nabla$ is $\Delta + \divergence(\cdot \nabla U)$,
        in that case, the hydrodynamic limit is given by the Fokker--Planck equation
        \begin{equation}\label{Fokker}
        \partial_t\rho_t = C_k\nabla \cdot \left(\nabla \rho_t + \rho_t\nabla U\right)
        \end{equation}
        
        \item If instead of \eqref{eq:symmetric weights for thm} we consider the modified symmetric weights $W_N$ as in \eqref{eq:kernel symmetric weight} below, then we obtain a variant of our main result where the hydrodynamic limit is given by
    \begin{equation*}
        \partial_t\rho_t = C_{k,\alpha}e^{(2\alpha - 1)U}\left(\Delta - \alpha\nabla U \cdot \nabla\right)\rho_t.
    \end{equation*}
    \end{enumerate}
\end{remark}

\section{Existence of the SEP}\label{sec:Existence}
In a non-compact setting, the existence of the SEP is not trivial, because the weighted graphs $G_N$ with weights $W_N$
given in \eqref{eq:symmetric weights for thm} are infinite graphs as soon as the measure $d\mu$ in \eqref{refmeas} is an infinite measure (i.e., when $\mu(M)=\infty$). Even if  the sum of jump rates is locally finite; that is, $\sum_{y \sim x} W(x,y) < +\infty$ almost surely for all $x \in \Lambda^N$, with our choice of weights we have
\begin{equation}\label{eq:unif}
    \sup_{x \in \Lambda^N} \sum_{y \sim x} W(x,y) = +\infty \quad \text{a.s.}
\end{equation}
whenever $\mu(M)=\infty$.
Hence, the existence criterion \cite[Proposition 6.1]{liggett_interacting_2005} for the process is not satisfied.

Instead, via the graphical construction of (SEP), it is enough to prove that a single random on the graph has infinite existence time for every starting point. To this end, we rely on a general criterion of graph stochastic completeness.
\begin{theorem}[Folz \cite{folz_volume_2013}]\label{thm:Folz}
    Let $(G, W)$ be a countably finite, locally finite connected weighted graph and $d_G$ a graph metric such there exist constants $C, C' > 0$ satisfying
    \begin{equation}\label{eq:criterion C C prime}
        d_G(x, y) \le C \text{ if } x \sim y, \quad \sum_{y \sim x} W(x,y)\, d_G(x,y) ^2 \le C' \text{ for all } x \in G.
    \end{equation}
    Then $(G, W)$ is stochastically complete if 
    \begin{equation}\label{eq:Osgood}
        \int_{r_o}^{\infty} \frac{r}{\log \# B_{G}(o, r)}\, dr = +\infty
    \end{equation}
    for some $o \in G$ and $r_o > 0$, where $B_{G}(o, r)$ denotes the graph ball of radius $r$ centered at $o$.
\end{theorem}
The Ricci curvature lower bound $\Ric_M \ge (m-1)\kappa$ we assume allows us to compare the volume of large balls with the ones of the model spaces.
Ultimately, thanks to this comparison, we will be able to upper bound the denominator of \eqref{eq:Osgood} in a way that the integral diverges, and thus prove the existence of the process.
For the comparison of  the volume of large balls with the ones of the model spaces, we need the Bishop--Gromov inequality stated below.
\begin{theorem}[{Bishop--Gromov inequality \cite[Theorem 11.19]{lee_introduction_2018}}]\label{thm:Bishop-Gromov inequality}
    Assume $\Ric_M \ge (m-1)\kappa$ for some $\kappa \in \R$.
    Then, for every $x \in M$,
    \begin{equation}\label{eq:Bishop Gromov ineq}
        \Vol_M(B_M(x,r)) \le {\Vol_\kappa(r)},
    \end{equation}
    where $\Vol_\kappa(r)$ denotes the volume of a ball of radius $r$ in the simply connected space of constant sectional curvature $\kappa$.
\end{theorem}
\begin{remark}\label{rmk:growth}
    Under the Ricci lower bound $\Ric_M \ge (m-1)\kappa$, if $\kappa > 0$, then the manifold is compact by Myer's theorem \cite[Theorem 12.24]{lee_introduction_2018}.
    If $\kappa = 0$, then its volume growth is at most like $r^m$. 
    Finally, if $\kappa < 0$, then its volume growth is at most like $e^{(m-1)\sqrt{-\kappa}r}$.
\end{remark}
We also need a basic large deviation estimate for the Poisson distribution stated below.
\begin{proposition}[{Bennett's inequality \cite[Theorem 2.9]{boucheron_concentration_2013}}]\label{thm:Bennett's inequality}
    One has
    \begin{equation}\label{eq:Bennett upper}
        \P\left[\mathrm{Poisson}(\lambda) \ge \lambda\left(1 + s)\right)\right] \le \exp\left(-\lambda\Phi(s)\right)
    \end{equation}
    for $s \ge 0$ and 
    \begin{equation*}
        \P\left[\mathrm{Poisson}(\lambda) \le \lambda\left(1 + s)\right)\right] \le \exp\left(-\lambda\Phi(s)\right)
    \end{equation*}
    for $-1 \le s \le 0$, where $\Phi(s) \coloneqq (1 + s)\log(1 + s) - s$.
\end{proposition}
\begin{theorem}[Quenched existence of the SEP] \label{thm:existence}
    Let $(G_N=(V_N,E_N))_{N \ge 1}$ be a sequence of $\mathrm{RNG}$ on $M$ in the sense of \cref{def:Random neighbourhood graph} with symmetric weights $(W_N)_{N\ge1}$ given by  \eqref{eq:symmetric weights for thm}. Assume that $\Ric_M \ge (m-1)\kappa$ for some $\kappa \in \R$ and that $e^{-U}$ is bounded. Then for almost every realization of the sequence of graphs, the $\mathrm{SEP}$ exists for all time.
\end{theorem}
\begin{proof}
\textbf{Step 1 (Choice of an adapted graph metric)}.
The existence is trivial if $M$ is compact. 
Thus we assume $\Ric_M \ge (m-1)\kappa$ for some $\kappa \le 0$ in the non-compact case; recall that $\kappa > 0$ implies that $M$ is compact.
Set the length of edges as
\begin{equation*}
    l_{G_N}(x, y) \coloneqq \mathrm{min}\left\{1, \frac{1}{\sqrt{\lambda_N(x) \vee \lambda_N(y)}}\right\}, \quad \lambda_N(x) \coloneqq \sum_{y \sim x} W_N(x, y),
\end{equation*}
and let $d_{G_N}$ be the associated shortest-path metric; that is,
\begin{equation*}
    d_{G_N}(x, y) \coloneqq \inf\left\{\sum_{i=1}^\ell l_{G_N}(X^i, X^{i+1}) \: : \: X_1 = x, X_{\ell+1} = y \text{ and } X_i \sim X_{i+1} \right\}.
\end{equation*}
Then we can pick $C = C' = 1$ in \eqref{eq:criterion C C prime} for any connected component of the graph since $d_{G_N}(x, y) \le 1$ for $x \sim y$ and
\begin{equation*}
    \sum_{y \sim x} W_N(x,y)\, d_{G_N}(x,y)^2 \le \sum_{y \sim x} W_N(x,y) \,\frac{1}{\lambda_N(x)} = 1.
\end{equation*}
\noindent \textbf{Step 2 (Comparison of the graph and Riemannian balls)}.
We claim that for all $x \in \Lambda^N$,
\begin{equation}\label{eq:claim}
    B_{{G_N}}(x, r) \subseteq \Lambda_N \cap B_M(x, C_1 r^2) \quad \text{for all $r$ large enough a.s.}
\end{equation}
with some constant $C_1(N,m,\kappa,\norm{k}_\infty, \norm{e^{-U}}_\infty)$. Let $y \in B_{G_N}(x, r)$. Then there exists a path $\gamma$ from $x$ to $y$ of the form
\begin{equation*}
    x \eqqcolon X^1 \sim \cdots \sim X^{\ell+1} \coloneqq y
\end{equation*}
for some $\ell > 0$ such that
\begin{equation}\label{eq:path bound}
    \sum_{i=1}^\ell l_N(X^i, X^{i+1}) \le r.
\end{equation}
Moreover, $X^i \in \overline{B_M}(x, R_\gamma)$ for all $1 \le i \le \ell+1$, where $R_\gamma \coloneqq \max \{d_M(x, X^i)\}_{i=1}^{\ell+1}$; the closure is with respect to the Riemannian distance. In particular,
\begin{equation}\label{eq:first inclusion}
    y \in \Lambda^N \cap \overline{B_M}(x, R_\gamma)
\end{equation}
and
\begin{equation}\label{eq:bound lambda lambda prime}
    \lambda_N(y) \le \lambda_N'(R_\gamma), \quad \lambda'_N(R) \coloneqq \sup_{z \in \Lambda^N \cap \overline{B_M}(x, R)} \lambda_N(z).
\end{equation}
Hence, combining \eqref{eq:path bound} and \eqref{eq:bound lambda lambda prime},
\begin{equation*}
    \frac{\ell}{\sqrt{\lambda'_N(R_\gamma + \varepsilon)}} \le \sum_{i=1}^\ell l_N(X^i, X^{i+1}) < r.
\end{equation*}
But since $R_\gamma \le \ell h_N$ by construction, we have
\begin{equation}\label{eq:ineq R gamma}
    R_\gamma < h_N r\sqrt{\lambda'_N(R_\gamma)}.
\end{equation}
Set
\begin{equation}\label{eq:R r epsilon}
    R(r) \coloneqq \sup \left\{r' > 0 \: : \: r' \le h_N r \sqrt{\lambda'_N(r')} \right\}.
\end{equation}
Then \eqref{eq:ineq R gamma} implies $R_\gamma \le R(r)$ so that
\begin{equation}\label{eq:second inclusion}
    \Lambda^N \cap \overline{B_M}(x, R_\gamma) \subseteq \Lambda^N \cap \overline{B_M}(x, R(r)).
\end{equation}
Combining the two inclusions \eqref{eq:first inclusion} and \eqref{eq:second inclusion}, it holds
\begin{equation*}
    B_{G_N}(x, r) \subseteq \Lambda^N \cap \overline{B_M}(x, R(r)).
\end{equation*}
We still have to obtain the almost sure upper bound $R(r) \lesssim r^2$ to prove the claim \eqref{eq:claim}. To this end, we estimate the probability of occurrence of large values of $\lambda_N'$ which then yields the desired upper bound.

Define the number of bad samples (with too much neighbors) as
\begin{equation*}
    N_{\mathrm{bad}}(R; K) \coloneqq \#\left\{z \in \Lambda^N \cap \overline{B_M}(x, R) \: : \: \lambda_N(z) \ge \norm{W_N}_\infty(1 + K) \right\},
\end{equation*}
for which
\begin{equation}\label{eq:simple bound in terms of bad sets}
    \P^x \left[\lambda'_N(R) \ge \norm{W_N}_\infty(1 + K)\right] \le \P^x\left[N_{\mathrm{bad}}(R; K) \ge 1\right],
\end{equation}
where $\P^x$ denotes the conditional probability on the event $x \in \Lambda^N$.
By conditional Markov's inequality,
\begin{align*}
    \P^x\left[N_{\mathrm{bad}}(R; K) \ge 1\right] &\le \E^x\left[N_{\mathrm{bad}}(R; K)\right] \\
    &= \int_{\overline{B_M}(x, R)} \, \P^z\left[\lambda_N(z) \ge \norm{W_N}_\infty(1 + K)\right] Ne^{-U(z)} d\Vol_M(z) \\
    &\le N\norm{e^{-U}}_\infty \left(\sup_{z \in \overline{B_M}(x, R)} \P^z\left[\lambda_N(z) \ge \norm{W_N}_\infty(1 + K)\right]\right) \times \Vol_\kappa(R), \numberthis\label{eq:bound bad in terms of lambda}
\end{align*}
where in the last line we used the Bishop--Gromov inequality \eqref{eq:Bishop Gromov ineq}.
Note that for $z \in \Lambda^N$,
\begin{equation*}
    \lambda_N(z) = \sum_{z' \sim z} W_N(z, z') \le \norm{W_N}_\infty \#\left(\Lambda^N \cap \overline{B_M}(z, h_N)\right).
\end{equation*}
Thus, for every $K$ large enough, combining the Bishop--Gromov inequality, which yields $\mu[\overline{B_M}(z, h_N)] \le \norm{e^{-U}}_\infty \Vol_\kappa(h_N)$, with Bennett's inequality \eqref{eq:Bennett upper},
\begin{align*}
    \P^z\left[\lambda_N(z) \ge \norm{W_N}_\infty(1 + K)\right] &\le \P^z\left[\#\left(\Lambda^N \cap \overline{B_M}(z, h_N)\right) \ge (1 + K) \right] \\
    &= \P\left[\mathrm{Poisson}\left(N\mu\left[\overline{B_M}(z, h_N)\right]\right) \ge K \right] \\
    &\le e^{-C_2K\log K} \numberthis\label{eq:Poisson tail}
\end{align*}
for some constant $C_2 > 0$ depending only on $N, m, \kappa, \norm{k}_\infty$ and $\norm{e^{-U}}_\infty$. Inserting \eqref{eq:Poisson tail} into \eqref{eq:bound bad in terms of lambda} and recalling \eqref{eq:simple bound in terms of bad sets} gives
\begin{equation}\label{eq:summable}
    \P^x \left[\lambda'_N(R) \ge \norm{W_N}_\infty(1 + K)\right] \le C_3e^{-C_2K\log K} \max\left\{R^m,\, e^{(m-1)\sqrt{-\kappa} R}\right\}
\end{equation}
for $R$ large enough and some constant $C_3(N, m, \kappa, \norm{k}_\infty, \norm{e^{-U}}_\infty)$; we used the volume growth of $\Vol_\kappa(R)$. Now, setting $\lceil R \rceil \coloneqq K$ makes \eqref{eq:summable} summable in $K$;
\begin{equation*}
    \sum_{K = 1}^\infty \P^x \left[\lambda'_N(K) \ge \norm{W_N}_\infty(1 + K)\right] < +\infty.
\end{equation*}
Hence, the Borel--Cantelli lemma implies that
\begin{equation}\label{eq:bound lambda prime Borel Cantelli}
    \lambda'_N(K) \le \norm{W_N}_\infty(1 + K) \quad \text{for all $K$ large enough a.s.}
\end{equation}
Recall definition \eqref{eq:R r epsilon},
\begin{equation*}
    R(r) \le h_N r \sqrt{\lambda'_N(R(r))}.
\end{equation*}
Using the estimate \eqref{eq:bound lambda prime Borel Cantelli} on $\lambda'_N$, we obtain that almost surely $R(r) < C_1 r^2$ for $r$ large enough, proving the claim \eqref{eq:claim}.
\newline

\noindent \textbf{Step 3 (Verification of the stochastic completeness criterion)}. 
Since we have the inclusion
\begin{equation*}
    B_{{G_N}}(x, r) \subseteq \Lambda_N \cap B_M(x, C_1 r^2) \quad \text{for all large $r$ enough a.s.},
\end{equation*}
then for some large $r_x > 0$, almost surely
\begin{equation*}
    \int_{r_x}^{\infty} \frac{r}{\log \# B_{G_N}(x, r)}\, dr \ge \int_{r_x}^{\infty} \frac{r}{\log \# \left(\Lambda^N \cap B_{M}(x, C_1r^2)\right)}\, dr.
\end{equation*}
By the Bishop--Gromov inequality \eqref{eq:Bishop Gromov ineq},
\begin{align*}
    \P^x\Big[\# &\left(\Lambda^N \cap B_{M}(x, C_1r^2)\right) \ge 1 + \left(1 + \delta\right) N \norm{e^{-U}}_\infty\Vol_\kappa(C_1 r^2) \Big] \\
    &\le \P\Big[\mathrm{Poisson}\left(N \norm{e^{-U}}_\infty \Vol_M\left[B_{M}(x, C_1r^2)\right]\right) \ge \left(1 + \delta\right) N \norm{e^{-U}}_\infty\Vol_\kappa(C_1 r^2)\Big] \\
    &\le \P\Big[\mathrm{Poisson}\left(N \norm{e^{-U}}_\infty \Vol_\kappa(C_1 r^2)\right) \ge \left(1 + \delta\right) N \norm{e^{-U}}_\infty\Vol_\kappa(C_1 r^2)\Big],
\end{align*}
and via Bennett's inequality \eqref{eq:Bennett upper}, for every $\delta > 0$, the last term can be estimated as
\begin{equation}\label{eq:Bennett summable}
    \P\Big[\mathrm{Poisson}\left(N \norm{e^{-U}}_\infty \Vol_\kappa(C_1 r^2)\right) \ge \left(1 + \delta\right) \Vol_\kappa(C_1 r^2)\Big] \le \exp\left(- N \norm{e^{-U}}_\infty\Vol_\kappa(C_1 r^2) \Phi(\delta)\right).
\end{equation}
Since $\Vol_\kappa(C_1 r^2) \gtrsim r^{2m}$ for $r > 1$, the right-hand side of \eqref{eq:Bennett summable} is summable in $\lceil r \rceil$ (recall that $N$ is fixed), and the Borel--Cantelli lemma yields
\begin{equation*}
    \# \left(\Lambda^N \cap B_{M}(x, C_1r^2)\right) \le N \norm{e^{-U}}_\infty \Vol_\kappa(C_1 r^2) \le C_4 e^{C_5 r^2} \quad \text{for all $r$ large enough a.s.}
\end{equation*}
and constants $C_4, C_5 > 0$ depending only on $N,m,\kappa,\norm{k}_\infty$, and $\norm{e^{-U}}_\infty$ letting $\delta \downarrow 0$.
Finally, we obtain almost surely
\begin{equation*}
    \int_{r_x}^{\infty} \frac{r}{\log \# B_{G_N}(x, r)}\, dr \ge \int_{r_x}^{\infty} \frac{r}{\log \# \left(\Lambda^N \cap B_{M}(x, C_1r^2)\right)}\, dr \gtrsim \int_{r_x}^{\infty} \frac{1}{r}\, dr = +\infty.
\end{equation*}
The stochastic completeness of the weighted random walk in the connected component of $x \in \Lambda^N$ follows by \cref{thm:Folz}, which implies the existence of the SEP in that connected component by the graphical construction.
Since the above argument works for all $x \in \Lambda^N$, we obtain the global existence of the SEP which concludes the proof.
\end{proof}
\begin{remark}
    A commonly used road to show the existence of the (SEP) is by a percolation argument, as in 
    \cite{harris}, which boils down to show the finiteness of ``active clusters'' of a site $x$, uniformly in $x$. In our setting this boils down to adapt arguments from \cite{meester1996continuum} for the Poisson--Boolean model. This methodology leads, however, to additional requirements (upper bounds) on the (sectional) curvature. Therefore, we have chosen the route of showing stochastic completeness instead.
\end{remark}
\section{Abstract hydrodynamic limit via duality}\label{Sec:Abstract hydrodynamic limit via duality}
The SEP satisfies a duality property with a single random walk due to the symmetry of the jump rates.
In words, this property means that 
the SEP generator acting on the $\eta$ variable of the function $D(x,\eta)=\eta(x)$ has the same effect as the single-random walk generator acting on the
$x$ variable of that same function (which is therefore named single particle duality function).
\begin{definition}[Random walk (RW)]\label{def:Random walk}
    Given a graph $G = (V,E)$ with weights $W$, the \emph{random walk (RW)} is the continuous-time Markov jump processes on $V$ whose generator evaluated on functions 
    $\varphi : V \to \R$ is given by
    \begin{equation}\label{eq:RW generator}
        L_{G}^{\RW} \varphi(x) \coloneqq \sum_{y \sim x} W(x,y)\left(\varphi(y) - \varphi(x)\right).
    \end{equation}
\end{definition}
The duality between the RW and the SEP described above manifests itself as follows: for pairings between the empirical measure of a configuration $\eta$ and a local function $\varphi$ on $V$, it holds
\begin{equation*}
    L_G^{\SEP}\left\langle \varphi, \sum_{x \in V}\eta(x)\delta_x \right\rangle = \left\langle L_G^{\RW} \varphi, \sum_{x \in V}\eta(x)\delta_x\right\rangle.
\end{equation*}
Before stating and proving our abstract hydrodynamic limit for the SEP via duality, we introduce three assumptions on the sequence of graphs $(G_N = (V_N, E_N))_{N \ge 1}$ with symmetric weights $(W_N)_{N \ge 1}$ and bandwidth parameters $(h_N)_{N \ge 1} \downarrow 0$:
\begin{itemize}
    \item[(i)] (Initial data is well-prepared). Assume that the empirical measure $\widehat{\pi_0}^N \coloneqq \tfrac{1}{N}\sum_{x \in V_N} \eta^N_0(x) \delta_x$ associated to the initial configuration $\eta^N_0$ of the $\mathrm{SEP}$ on $G_N$ converges vaguely to $\rho_0 \,d\Vol_M$ in probability; that is,
    \begin{equation}\label{eq:initial well prepared}\tag{\MakeUppercase{\romannumeral 1}}
        \forall \varphi \in C_0(M), \quad \left\langle \varphi, d\widehat{\pi_0}^N \right\rangle \overset{\mathrm{pr.}}{\to} \left\langle \varphi, \rho_0 \: d\Vol_M\right\rangle \quad \text{as }N\to +\infty;
    \end{equation}
    \item[(ii)] (Upper local density). For every compact set $S\Subset M$ there exist constants $r_S, C_S > 0$ such that, for all
    $N$ large enough,
    \begin{equation}\label{eq:local assumption}\tag{\MakeUppercase{\romannumeral 2}}
        \sup_{x\in S}\ \sup_{h_N\le r\le r_S}\ 
    \frac{\#(V_N\cap B_M(x,r))}{N\,\Vol_M(B_M(x,r))}\ \le\ C_S;
    \end{equation}
    \item[(iii)] (Convergence of the generator). Let $\L$ be the generator of a diffusion process on $M$ whose core is $C_c^\infty(M)$. 
    Assume that the diffusively rescaled generator $h_N^{-2} L_{G_N}^{\RW}$ converges towards the generator $\L$ in the following sense
    \begin{equation}\label{eq:assumption convergence laplacian}\tag{\MakeUppercase{\romannumeral 3}}
        \forall \varphi \in C^\infty_c(M), \quad \lim_{N \to \infty} \frac{1}{N}\sum_{x\in V_N} \abs{h_N^{-2} L_{G_N}^{\RW} \varphi(x) - \L \varphi(x)} = 0.
    \end{equation}
\end{itemize}
\begin{theorem}(Abstract hydrodynamic limit)\label{thm:Abstract hydrodynamic limit}
    Let $(G_N = (V_N, E_N))_{N \ge 1}$ be a sequence of graphs on $M$ with symmetric weights $(W_N)_{N\ge1}$ and bandwidth parameters $(h_N)_{N \ge 1}$.
    If the assumptions \eqref{eq:initial well prepared}, \eqref{eq:local assumption}, and \eqref{eq:assumption convergence laplacian} are satisfied, then the trajectory of the empirical measure $t \mapsto \widehat{\pi_t}^N$ associated to a configuration $\eta_t^N$ that evolves according to the diffusively rescaled $\mathrm{SEP}$ with generator $h_N^{-2} L_{G_N}^{\SEP}$ converges in the Skohorod topology to a distribution $t\mapsto \rho_t\, d\Vol_M$ whose density solves weakly the forward Kolmogorov equation 
    \begin{equation}\label{eq:forward Kolmogorov}
        \partial_t \rho_t = \L^* \rho_t
    \end{equation}
    with initial condition $\rho_0$ of assumption \eqref{eq:initial well prepared}, provided there exists a unique weak solution; the adjoint is understood in $L^2(M, d\Vol_M)$.
\end{theorem}
\begin{proof}
    We outline the argument for complete manifolds modified from the torus case found in the monograph of Kipnis and Landim \cite[Chapter 4]{kipnis_scaling_1999}.

    \textbf{Step 1 (The Dynkin martingale argument)}.
    We pair any test function $\varphi \in C_c^\infty(M)$ with the empirical measure $\widehat{\pi_t}^N$ associated to the configuration $\eta_t$ (with $0 \le t \le T$ for some time $T > 0$), and we consider its associated Dynkin martingale;
    \begin{equation*}
        M_t^N\left[\varphi\right] = \left\langle \varphi, \widehat{\pi_t}^N \right\rangle - \left\langle \varphi, \widehat{\pi_0}^N \right\rangle - \int_0^t h_N^{-2} L_{G_N}^{\SEP} \left\langle \varphi, \widehat{\pi_s}^N \right\rangle \: ds.
    \end{equation*}
    By duality between the SEP and the RW,
    \begin{equation*}
         h_N^{-2} L_{G_N}^{\SEP}\left\langle \varphi, \widehat{\pi_s}^N\right\rangle = \left\langle h_N^{-2} L_{G_N}^{\RW} \varphi, \widehat{\pi_s}^N \right\rangle.
    \end{equation*}
    Therefore, convergence of generators can be established at the level of random walks for a single particle, which is guaranteed by assumption \eqref{eq:assumption convergence laplacian}, and the above yields
    \begin{equation*}
        M_t^N\left[\varphi\right] = \left\langle \varphi, \widehat{\pi_t}^N \right\rangle - \left\langle \varphi, \widehat{\pi_0}^N \right\rangle - \int_0^t \left\langle \L \varphi, \widehat{\pi_s}^N \right\rangle\: ds  + o(1).
    \end{equation*}
    To show that the Dynkin martingale goes to 0 in probability, it is easier to consider its quadratic variation given by
    \begin{equation*}
        \left[M_t^N, M_t^N\right]\left(\varphi\right) = \int_0^t h_N^{-2} L_{G_N}^{\SEP} \left\langle \varphi, \widehat{\pi_s}^N \right\rangle^2 - 2\left\langle \varphi, \widehat{\pi_s}^N \right\rangle h_N^{-2} L_{G_N}^{\SEP} \left\langle \varphi, \widehat{\pi_s}^N \right\rangle \: ds.
    \end{equation*}
    Again, since both the rates of the jumps are symmetric and the generator $h_N^{-2} L_{G_N}^{\RW}$ converges towards $\L$, some arithmetic manipulations yield
    \begin{equation*}
        \left[M_t^N, M_t^N\right]\left(\varphi\right) \le \frac{1}{N^2}\sum_{y \in V_N} \abs{ \varphi(y) L_N^{\RW} \varphi(y) } = \frac{1}{N^2}\sum_{y \in V_N} \abs{ \varphi(y) \L \varphi(y)} + o(1) = o(1)
    \end{equation*}
    where in the last equality we used assumption \eqref{eq:local assumption}; $\#(\Supp\,\varphi \cap V_N)$ is asymptotically of order at most $N$ by compactness of the support of $\varphi$.
    By Doob's inequality,
    \begin{equation*}
    \E\left[\sup_{0\le t\le T}|M_t^N[\varphi]|^2\right]\le 4\,\E\left[\left[M_t^N, M_t^N\right]\left(\varphi\right)\right] \to 0,
    \end{equation*}
    hence $\sup_{t\le T}|M_t^N[\phi]|\to 0$ in probability; that is,
    \begin{equation*}
        \sup_{0 \le t\le T} \abs{\left\langle \varphi, \widehat{\pi_t}^N \right\rangle - \left\langle \varphi, \widehat{\pi_0}^N \right\rangle - \int_0^t \left\langle \L \varphi, \widehat{\pi_s}^N \right\rangle\: ds} \overset{\mathrm{pr.}}{\to} 0 \quad \text{at } N \to +\infty.
    \end{equation*}
    
    \textbf{Step 2 (Relative compactness of the associated measure)}.
    We equip the space of càdlàg paths with the probability measure associated to $t \mapsto \widehat{\pi_t}^N$ that we denote by $Q^N$.
    Then the convergence obtained in Step 1 implies that
    \begin{equation*}
        \lim_{N\to+\infty} Q^N \left[ \left\{ \nu \in D([0, T]) : \sup_{0 \le t\le T } 
        \abs{\left\langle \varphi, \nu_t \right\rangle - \left\langle \varphi, \nu_0 \right\rangle - \int_0^t \left\langle \L \varphi, \nu_s \right\rangle\: ds} \le \epsilon \right \} \right] = 1
    \end{equation*}
    for all $\epsilon > 0$, where $D([0, T])$ denotes the Skorokhod path space of càdlàg trajectories. 
    It remains to show relative compactness of the associated measures $Q^N$. 
    
    Since we work in the vague topology on measures, it is sufficient to prove relative compactness of the projected measures, namely, of the sequence of real-valued random trajectories
    $\{\langle\pi^N_t, \varphi\rangle, 0\leq t\leq T\}$,
    for every smooth compactly supported test function $\varphi \in C_c^\infty(M)$. 
    Relative compactness of this sequence of trajectories is in turn a consequence of Aldous' criterion;
    \begin{equation*}
        \lim_{\Theta \to 0} \limsup_{N\to +\infty} \sup_{\tau \in \I_T, \theta \le \Theta} Q^N \left[\left\{\widehat{\pi}^N : \abs{   \left\langle \varphi, \widehat{\pi_{(\tau+\theta) \wedge T}}^N\right\rangle - \left\langle \varphi, \widehat{\pi_\tau}^N\right\rangle} > \epsilon \right\}\right] = 0,
    \end{equation*}
    where $\I_T$ is the set of all stopping times bounded by $T$. To verify Aldous's criterion, we rewrite the difference of empirical measure using the Dynkin martingale;
    \begin{equation*}
        \left\langle \varphi, \widehat{\pi_{(\tau+\theta) \wedge T}}^N\right\rangle - \left\langle \varphi, \widehat{\pi_\tau}^N\right\rangle = M_{(\tau+\theta) \wedge T}^N\left[\varphi\right] - M_\tau^N\left[\varphi\right] + \int_\tau^{(\tau+\theta) \wedge T} \left\langle \L\varphi , \widehat{\pi_s}^N\right\rangle \: ds + o(1).
    \end{equation*}
    The integral is of order $\theta$. 
    For the difference of the martingales, Chebyshev's inequality yields
    \begin{align*}
        Q^N \left[\left\{\widehat{\pi}^N : \abs{ M_{(\tau+\theta)\wedge T}^N\left[\varphi\right] - M_{\tau}^N\left[\varphi\right] } > \epsilon \right\}\right] &\le \frac{1}{\epsilon^2}\E\left[ \abs{ M_{(\tau+\theta)\wedge T}^N\left[\varphi\right] - M_{\tau}^N\left[\varphi\right] }^2\right] \\
            &= \frac{1}{\epsilon^2}\E\left[ \left[M_{(\tau+\theta)\wedge T}^N, M_{(\tau+\theta)\wedge T}^N\right]\left(\varphi\right) - \left[M_{\tau}^N, M_{\tau}^N\right]\left(\varphi\right)\right] \\
            &= o(1).
    \end{align*}
    
    \textbf{Step 3 (Uniqueness of the limit)}. 
    Now that the relative compactness of the sequence of measures $\{Q^N\}_{N \ge 1}$ is established, one argues that all the limits of the weakly convergent subsequences are the same, whence $Q^N$ converges weakly to a limiting measure $Q$. 
    This limiting measure is concentrated on paths that weakly solve the forward Kolmogorov equation. 
    To see this, recall the closed set
    \begin{equation*}
        H_\epsilon \coloneqq \left\{ \nu \in D([0, T]) : \sup_{0 \le t\le T } 
        \abs{\left\langle \varphi, \nu_t \right\rangle - \left\langle \varphi, \nu_0 \right\rangle - \int_0^t \left\langle \L \varphi, \nu_s \right\rangle\: ds} \le \epsilon \right \}
    \end{equation*}
    and pick any converging subsequence $\{Q^{N_m}\}_{m \ge 1}$. Then 
    \begin{equation*}
        Q[H_\epsilon] \ge \limsup_{m \to +\infty} Q^{N_m}[H_\epsilon] = 1,
    \end{equation*} and it follows that
    \begin{equation*}
        Q\left[H_0\right] = Q\left[\cap_{m \ge 1} H_{1/m}\right] = 1 - Q\left[\cup_{m \ge 1} H_{1/m}^C \right] = 1,
    \end{equation*}
    which in turn implies that
    \begin{equation*}
        Q\left[\left\{ \nu \in D([0, T]) : \sup_{0 \le t\le T } 
        \abs{\left\langle \varphi, \nu_t \right\rangle - \left\langle \varphi, \nu_0 \right\rangle - \int_0^t \left\langle \L \varphi, \nu_s \right\rangle\: ds} = 0 \quad \forall \varphi \in C_c^\infty(M)\right\}\right] = 1.
    \end{equation*}
    Moreover, for $Q$-a.e. path $\pi=\{\pi_t\}_{t\in[0,T]}$, $\pi_t \ll \Vol_M$, namely, $\pi_t$ is absolutely continuous with respect to the volume measure.
    Indeed, by assumption \eqref{eq:local assumption}, for any compact $S \Subset M$, there exist $r_S,C_S>0$ such that for all $N$ large enough, for all $x\in S$ and $h_N <r\le r_S$,
    \begin{equation*}
        \widehat{\pi_t}^N \left[B_M(x,r)\right]\le \frac1N \#(V_N\cap B_M(x,r))\le C_S\,\Vol_M(B_M(x,r)).
    \end{equation*}
    Now, let $A\subset S$ be a Borel set with $\Vol_M(A)=0$. 
    Then by outer regularity, for every $\varepsilon>0$ there exists
    a countable cover $A\subset \bigcup_i B_M(x_i,r_i)$ with $r_i\le r_S$ and $\sum_i \Vol_M(B_M(x_i,r_i))<\varepsilon$.
    For $N(r_i)$ large enough, $h_N \le r_i \le r_S$ as $h_N \downarrow 0$, and hence for all $N \ge N(r_i)$,
    \begin{equation*}
    \widehat{\pi_t}^N\left[B_M(x_i,r_i)\right]\le C_S \Vol_M(B_M(x_i,r_i)).
    \end{equation*}
    By the Portmanteau theorem, this inequality passes to the limit;
    \begin{equation*}
        \pi_t\left[B_M(x_i, r_i)\right] \le \liminf_{N \to \infty} \widehat{\pi_t}^N\left[B_M(x_i,r_i)\right] \le C_S \Vol_M\left(B_M(x_i, r_i)\right).
    \end{equation*}
    In particular,
    \begin{equation*}
        \pi_t\left[A\right] \le \sum_{i} \pi_t\left[B_M(x_i, r_i)\right] \le C_S \sum_i \Vol_M\left(B_M(x_i, r_i)\right) \le C_S \varepsilon
    \end{equation*}
    Since $\varepsilon$ is arbitrary, $\pi_t(A)=0$. Hence $\pi_t\ll \Vol_M$ on $S$. As the compact $S\Subset M$ is arbitrary,
    we conclude $\pi_t\ll \Vol_M$ on $M$ for each $t\in[0,T]$.

    Again via the Dynkin martingale representation, one obtains that the trajectory $t \mapsto \widehat{\pi_t}^N$ weakly solves the forward Kolmogorov equation with initial condition $\rho_0$ from assumption \eqref{eq:initial well prepared}. Moreover, this trajectory is continuous in our setting. 
\end{proof}
We now explain how  to go from  \cref{thm:Abstract hydrodynamic limit}
towards the main result
\cref{thm:Hydrodynamic limit}.

It is straightforward to check that assumption \eqref{eq:local assumption} holds for the RNG of \cref{def:Random neighbourhood graph}. 
Assumption \eqref{eq:assumption convergence laplacian} on the convergence of discrete Laplacian towards weighted Laplacian is the content of \cref{thm:almost sure convergence empirical URW to Delta 2(1-alpha) general} for RNG; the limiting diffusion operator is $\L = C_k(\Delta - \nabla U \cdot \nabla)$. The fact that $\mathcal L$ has $C_c^\infty(M)$ as core can be found in e.g. \cite[Theorem 2.2]{jorgenson_heat_2006}.
Assumption \eqref{eq:initial well prepared} is made w.r.t.\ the reference Gibbs measure 
$d\mu = e^{-U}d\Vol_M$ in \cref{thm:Hydrodynamic limit} instead of w.r.t.\  the volume measure $d\Vol_M$. Hence, the corresponding density w.r.t.
the volume measure $d\Vol_M$ is $\rho_t e^{-U}$. Therefore, 
\cref{thm:Abstract hydrodynamic limit}, in particular \eqref{eq:forward Kolmogorov},
gives 
\begin{equation*}
    \partial_t \left(\rho_t e^{-U}\right) = \mathcal{L}^* \left(\rho_t e^{-U}\right) = C_k\left(\Delta - \nabla U \cdot \nabla\right)^* \left(\rho_t e^{-U}\right) = C_k \nabla \cdot \Big(\nabla \left(\rho_t e^{-U}\right) + \left(\rho_t e^{-U}\right) \nabla U\Big),
\end{equation*}
where we used that the $L^2(M, d\Vol_M)$ adjoint of $\Delta - \nabla U \cdot \nabla$ is $\Delta + \divergence(\cdot \nabla U)$. 
In particular, we recover the hydrodynamic limit \eqref{eq:heat drift};
\begin{equation*}
    \partial_t \rho_t = C_k \left(\Delta - \nabla U \cdot \nabla\right) \rho_t.
\end{equation*}
The existence of the SEP is the content of \cref{thm:existence} and the uniqueness of weak solutions follows from the stochastic completeness of the weighted manifold.

\section{Convergence of discrete random walks}\label{Sec:Convergence of discrete random walks}
In the proof of \cref{thm:Abstract hydrodynamic limit}, a specific notion of convergence of the RW towards a limiting differential operator is used to rewrite (almost surely in the case of RNG)
\begin{equation}\label{eq:notion of convergence}
    h_N^{-2} L_{G_N}^{\SEP}\left\langle \varphi(x), \widehat{\pi_t}^N\right\rangle = \left\langle h_N^{-2} L_{G_N}^{\RW}\varphi(x), \widehat{\pi_t}^N\right\rangle = \left\langle \mathcal{L}\varphi(x), \widehat{\pi_t}^N\right\rangle + o(1).
\end{equation}
The almost-sure pointwise consistency of graph Laplacians was obtained by Hein, Audibert and von Luxburg \cite[Theorems 28 \& 30]{hein_graph_2007} for random neighbourhood graphs sampled from i.i.d. points (see also \cite[Theorem 2.1]{belkin_discrete_2008}). We obtain their counterpart on our RNG obtained from $\mathrm{PPP}$ with Gibbs intensity; the link resides in the fact that conditioned on the number of points in a region of finite measure, the points in that region are i.i.d. with respect to the restricted normalized measure \cite[Proposition 3.8]{last_lectures_2017}.

To perform the RW, we consider weights similar to \eqref{eq:symmetric weights for thm} with an additional parameter $\alpha \in \R$ given by
\begin{equation}\label{eq:kernel symmetric weight}
    W_{N;\alpha}(X^i, X^j) \coloneqq \frac{h_N^{-m} k \left(\frac{d_M(X^i, X^j)}{h_N}\right)}{\left(\overline{k}(X^i) \overline{k}(X^j)\right)^{\alpha}}, \quad \overline{k}(X^i) \coloneqq \frac{1}{N} \sum_{X^j \in \Lambda^N} h_N^{-m}k\left(\frac{d_\M(X^i, X^j)}{h_N}\right),
\end{equation}
where $\overline{k}$ plays the role of a statistical estimator of the Gibbs measure. The parameter $\alpha$ leads to variants of the weighted Laplacian $\Delta - \nabla U \cdot \nabla$.
\begin{definition}[Weighted Laplacian]\label{def:Weighted Laplacian}
    Given a smooth function $\phi \in C^\infty(\M)$, the \emph{$\alpha$-weighted Laplacian} is the reversible operator with respect to the $\alpha$-weighted Gibbs measure $d\mu_\alpha \coloneqq e^{-\alpha U} d\Vol_\M$; that is, the operator satisfying
    \begin{equation*}
        \forall \varphi \in C_c^\infty(\M), \quad \int_\M \varphi \Delta_\alpha \phi \:e^{-\alpha U}d\Vol_\M = -\int_\M \nabla \varphi \cdot \nabla \phi \:e^{-\alpha U}d\Vol_\M,
    \end{equation*}
    which is given by
    \begin{equation*}\label{eq:Weighted Laplacian Delta_alpha}
        \Delta_\alpha \coloneqq e^{\alpha U} \nabla \cdot \left(e^{-\alpha U} \nabla\right) = \Delta - \alpha \nabla U \cdot \nabla.
    \end{equation*}
\end{definition}
\begin{remark}
Hein, Audibert and von Luxburg consider the manifold $\M$ as a submanifold isometrically embedded in some euclidean space $\R^M$ using the fact that for close enough points, the intrinsic and extrinsic distances match up to the third order (\cite[Proposition 6]{smolyanov_chernoffs_2007}). We stick to the intrinsic setting for our purpose, namely, the hydrodynamic limit of the $\mathrm{SEP}$, and we show the consistency of the graph Laplacian with \eqref{eq:kernel symmetric weight} in our setting of $\mathrm{RNG}$ obtained from a $\mathrm{PPP}$.
\end{remark}
\begin{theorem}\label{thm:almost sure convergence empirical URW to Delta 2(1-alpha) general}
    Let $\varphi \in C_c^\infty(\M)$ and let $(G_N)_{N \ge 1}$ be a sequence of $\mathrm{RNG}$ in the sense of \cref{def:Random neighbourhood graph} with symmetric weights $(W_{N; \alpha})_{N\ge 1}$ given by \eqref{eq:kernel symmetric weight} and associated $\mathrm{RW}$ with diffusively rescaled generator $h_N^{-2} L_{G_N}^{\RW}$. 
    There are constants $C_{\alpha;1}, C_{\alpha;2} > 0$ that only depend on $\alpha, k, m, \mu$ and $\varphi$ and constants $C_{k,0}, C_{k,2} > 0$ that depend on the moments of $k$ such that, for all $\varepsilon > 0$ and $h_N$ small enough,
    \begin{multline*}
        \P\left[\sup_{X^i \in \Lambda^N} \abs{h_N^{-2} L_{G_N}^{\RW}\varphi(X^i) - \left(\frac{C_{k,2}}{2\left(C_{k,0}\right)^{2\alpha}}e^{(2\alpha-1)U(X^i)}\Delta_{2(1-\alpha)}\varphi(X^i) + O(h_N^2)\right)} > \varepsilon\right] \\
            \le C_{\alpha;1}N e^{-C_{\alpha;2} Nh_N^{m+2}\varepsilon^2}.
    \end{multline*}
    Then, if both $h_N \to 0$ and $Nh_N^{m+2}/\log N \to +\infty$ as $N \to +\infty$, it holds almost surely and uniformly in the vertices $X^i$,
    \begin{equation}\label{eq:right hand-side thm}
        \lim_{N \to +\infty} h_N^{-2} L_{G_N}^{\RW}\varphi(X^i) = \frac{C_{k,2}}{2\left(C_{k,0}\right)^{2\alpha}}e^{(2\alpha-1)U(X^i)}\Delta_{2(1-\alpha)}\varphi(X^i).
    \end{equation}
\end{theorem}
We will use the following concentration inequality in addition to \cref{thm:Bennett's inequality}:
\begin{proposition}[{Bernstein's inequality \cite[Theorem 2.10]{boucheron_concentration_2013}}]\label{thm:Bernstein's inequality}
    Let $X^1, \dots, X^n$ be $n$ independent random variables such that $\abs{X^i} \le b$ for some $b > 0$ a.s for all $i=1, \dots, n$ and $v \coloneqq \sum_i \Var[X^i]$. Then
    \begin{equation*}
        \P\left[\abs{\sum_{i=1}^n \left(X^i - \E[X^i]\right)} \ge s\right] \le 2\exp\left(-\frac{s^2}{2\left(v + bs/3\right)}\right).
    \end{equation*}
\end{proposition}
\begin{proof}[Proof of \cref{thm:almost sure convergence empirical URW to Delta 2(1-alpha) general}]
    Consider the case $\alpha > 0$; the case $\alpha \le 0$ is similar. Set
    \begin{equation*}
        Y(x; y) \coloneqq h_N^{-m} k\left(\frac{d_\M(x, y)}{h_N}\right), \quad \overline{k}(x) \coloneqq \frac{1}{N}\sum_{X^l \in \Lambda^N} Y(x; X^l),
    \end{equation*}
    and
    \begin{equation*}
         Z(x; y) \coloneqq h_N^{-2}\frac{Y(x; y)\left(\varphi(y) -  \varphi(x)\right)}{\left(\overline{k}(x)\overline{k}(y)\right)^\alpha}, \quad Z_\E(x; y) \coloneqq h_N^{-2}\frac{Y(x; y)\left(\varphi(y) -  \varphi(x)\right)}{\left(\E\Big[\overline{k}(x)\Big]\right)^\alpha\left(\E\Big[\overline{k}(y)\Big]\right)^\alpha}
    \end{equation*}
    together with
    \begin{equation*}
        \overline{Z}(x) \coloneqq \frac{1}{N} \sum_{X^l \in \Lambda^N} Z(x; X^l), \quad \overline{Z_\E}(x) \coloneqq \frac{1}{N} \sum_{X^l \in \Lambda^N} Z_\E(x; X^l).
    \end{equation*}
    Note that for $x \in \Lambda^N$,
    \begin{equation*}
        \overline{Z}(x) = h_N^{-2} L_{G_N}^{\RW}\varphi(x),
    \end{equation*}
    and using \cref{thm:convergence integral to weighted Laplacian}, we have
    \begin{equation*}
        \overline{Z_\E}(x) =  \frac{C_{k,2}}{2\left(C_{k,0}\right)^{2\alpha}}e^{(2\alpha-1)U(x)}\Delta_{2(1-\alpha)}\varphi(x) + O(h_N^2).
    \end{equation*}
    \indent \textbf{Step 1 (Concentration estimate for the kernel)}.
    Let $h_1 \le h_N$ without loss of generality. We are only interested in the Poisson points inside the $h_1$-thickening or $2h_1$-thickening of the support of $\varphi$ since the empirical quantities at stake are zero otherwise, where the $t$-thickening is defined as
    \begin{equation*}
        \left(\Supp \,\varphi\right)_{t} \coloneqq \Big\{z \in \M; \: d_\M(z, \Supp \,\varphi) \le t\Big\}.
    \end{equation*}
    Conditional on the number of points $\Lambda^N\vhh = n$, the $X^i$'s are i.i.d. \cite[Proposition 3.8]{last_lectures_2017} with respect to the probability measure
    \begin{equation*}
        d\mu_{\vhh} \coloneqq \frac{\Ind_{\vhh} d\mu}{\mu\left[\vhh\right]}.
    \end{equation*}
    Let $\delta > 0$ be a parameter that will be tuned in terms of $\varepsilon$ later on. By the law of total probability,
    \begin{equation}\label{eq:P1 P2 bound}
        \P \left[\sup_{x \in \Lambda^N|_{\vh}} \abs{\overline{k}(x) - \E\Big[\overline{k}(x)\Big]} > \delta\right] 
        \le \sum_{n=1}^\infty \left(P_1^{(n)} + P_2^{(n)}\right)\P\left[E_n\right] 
    \end{equation}
    with
    \begin{align*}
        P_1^{(n)} &\coloneqq \P\left[\sup_{x \in \Lambda^N|_{\vh}} \abs{\overline{k}(x) - \E_{X\sim\mu_{\vhh}}\left[\left(\frac{n}{N}\right)Y(x; X)\right]} > \frac{\delta}{2} \Bigg|\, E_n\right], \\
        P_2^{(n)} &\coloneqq \P\left[\sup_{x \in \Lambda^N|_{\vh}}  \abs{\E_{X\sim\mu_{\vhh}}\left[\left(\frac{n}{N}\right)Y(x; X)\right] - \E\Big[\overline{k}(x)\Big]} > \frac{\delta}{2}\Bigg|\, E_n\right],
    \end{align*}
    and where $E_n$ is the event of $\Lambda^N$ having $n$ points in $\vhh$.
    
    First, we bound $P_1^{(n)}\P[E_n]$. For this, note that for all $x \in \vh$,
    \begin{equation*}
        \overline{k}(x) = \frac{1}{n}\sum_{X^l \in \Lambda^N|_{\vhh}} \left(\frac{n}{N}\right)Y(x; X^l).
    \end{equation*}
    We estimate
    \begin{equation}\label{eq:Y^i(X^l) estimate}
        \abs{\left(\frac{n}{N}\right)Y(x; X^l)} \le \left(\frac{n}{N}\right)\norm{k}_{\infty} h_N^{-m} \coloneqq C_1\left(\frac{n}{N}\right) h_N^{-m}
    \end{equation}
    and
    \begin{align*}
        \Var_{X^l \sim \mu_{\vhh}} \left[\left(\frac{n}{N}\right)Y(x; X^l)\right] &\le \E_{X^l \sim \mu_{\vhh}} \left[\left(\frac{n}{N}\right)^2 \left(Y(x; X^l)\right)^2\right] \\
            &\le \left(\frac{n}{N}\right)^2 \frac{\norm{k^2 e^{-U}\big|_{\vhh}}_{\infty}}{\mu\left[(\Supp \,\varphi)_{2h_1}\right]^2} \sup_{z \in \left(\Supp\,\varphi\right)_{h_1}} \Vol_\M\left(B(z, h_N)\right) h_N^{-2m}\\
            &\le C_2 \left(\frac{n}{N}\right)^2 h_N^{-m} \numberthis\label{eq:Var Y^i(X^l)}
    \end{align*}
    for some constant $C_2(k, \mu, m, \varphi)$ where in the last inequality we used the fact that, since $\varphi$ is compactly supported, one can uniformly control the volume of balls in the $h_1$-thickening of its support;
    \begin{equation*}
        \Vol_\M(B(x, h_N)) = \Vol_{\R^m}(B(x,h_N))\left(1 + O(h_N^2)\right).
    \end{equation*}
     Then, using estimates \eqref{eq:Y^i(X^l) estimate} and \eqref{eq:Var Y^i(X^l)} for \cref{thm:Bernstein's inequality}, namely, Bernstein's inequality, together with the union bound, we obtain an exponential bound for $P_1^{(n)}$;
    \begin{multline}\label{eq:union bound P overline k Supp varphi conditional on n}
        \P\left[\sup_{x \in \Lambda^N|_{\vh}} \abs{\overline{k}(x) - \E_{X\sim\mu_{\vhh}}\left[\left(\frac{n}{N}\right)Y(x; X)\right]} > \frac{\delta}{2} \Bigg|\, E_n\right] \\
            \le 2n\exp\left(-\frac{n\, h_N^m \left(\frac{\delta}{2}\right)^2}{\frac{2}{3}C_1\left(\frac{n}{N}\right)\left(\frac{\delta}{2}\right) + 2C_2\left(\frac{n}{N}\right)^2}\right).
    \end{multline}
    Now, having obtained such exponential bound, we split $\sum_n P_1^{(n)} \P[E_n]$ in two:
    
    When $N_{-} \coloneqq N\mu[\vhh](1 - \tfrac{\delta}{C}) \le n \le N\mu[\vhh](1 + \tfrac{\delta}{C}) \eqqcolon N_{+}$ for some positive constant $C > 0$,
    \begin{equation}\label{eq:elementary exponential inequality}
        \exp\left(-\frac{n\, h_N^m\left(\frac{\delta}{2}\right)^2}{\frac{2}{3}C_1\left(\frac{n}{N}\right)\left(\frac{\delta}{2}\right) + 2C_2\left(\frac{n}{N}\right)^2}\right) \le \exp\left(-\frac{N h_N^m \left(\frac{\delta}{2}\right)^2}{\frac{2}{3}C_1\left(\frac{\delta}{2}\right) + 2C_2\mu\left[\vhh\right]\left(1 + \frac{\delta}{2C}\right)}\right).
    \end{equation}
    Since $\sum_n \mathbb{P}[E_n] = 1$,
    \begin{equation*}
        \sum_{n=\lceil N_{-} \rceil}^{\lfloor N_+ \rfloor} P_1^{(n)} \P[E_n] \le \max_{\lceil N_{-} \rceil \le n \le \lfloor N_+ \rfloor} P_1^{(n)},
    \end{equation*}
    and thus the exponential bound \eqref{eq:union bound P overline k Supp varphi conditional on n} combined with \eqref{eq:elementary exponential inequality} yields
    \begin{equation}\label{eq:P1 final bound part 1}
        \sum_{n=\lceil N_{-} \rceil}^{\lfloor N_+ \rfloor} P_1^{(n)} \P[E_n] \le C_3 N \exp\left(- C_4 Nh_N^m\, \delta^2\right)
    \end{equation}
    for $\delta$ small enough and some constants $C_3, C_4 > 0$ that depend only on $k, \mu, m,$ and $ \varphi$.
 
    When $n \notin [\lceil N_{-} \rceil, \lfloor N_+ \rfloor] $, the upper and lower tails of the Poisson random variable are controlled by \cref{thm:Bennett's inequality}, namely, Bennett's inequality;
    \begin{equation}\label{eq:upper and lower tail Poisson distribution via Bennett}
    \begin{aligned}
        \P\left[\Lambda^N(\vhh) > N\mu\left[\vhh\right]\left(1 + \frac{\delta}{C}\right)\right] &\le e^{-N\mu\left[\vhh\right]\Phi(\delta/C)}, \\
        \P\left[\Lambda^N(\vhh) < N\mu\left[\vhh\right]\left(1 - \frac{\delta}{C}\right)\right] &\le e^{-N\mu\left[\vhh\right]\Phi(\delta/C)}.
    \end{aligned}
    \end{equation}
    Therefore,
    \begin{equation}\label{eq:P1 final bound part 2}
        \left(\sum_{n=1}^{\lfloor N_-\rfloor} + \sum_{n=\lceil N_+ \rceil}^\infty\right) P_1^{(n)}\P\left[E_n\right] \le 2e^{-N\mu\left[\vhh\right]\Phi(\delta/C)}.
    \end{equation}
    Second, we bound $P_2^{(n)}\P[E_n]$. \cref{thm:lower bound upper bound E overline k} yields the upper bound $\E[\overline{k}(x)] \le C$ on $\vhh$,  
    and since
    \begin{equation*}
        \E_{X\sim\mu_{\vhh}}\left[\left(\frac{n}{N}\right)Y(x; X)\right] = \left(\frac{n}{N\mu\left[\vhh\right]}\right)\E\Big[\overline{k}(x)\Big],
    \end{equation*}
    we have
    \begin{align*}
        \P\Bigg[\sup_{x \in \Lambda^N|_{\vh}}\Big|&\E_{X\sim\mu_{\vhh}}\left[\left(\frac{n}{N}\right)Y(x; X)\right] - \E\Big[\overline{k}(x)\Big]\Big| > \frac{\delta}{2}\Bigg|\, E_n\Bigg]\P\left[E_n\right] \\
            &\le \P\Bigg[\sup_{x \in \Lambda^N|_{\vh}}\abs{1 - \left(\frac{n}{N\mu\left[\vhh\right]}\right)} > \frac{\delta}{2C} \Bigg|\, E_n\Bigg]\P\left[E_n\right] \\
            &\le 2e^{-N\mu\left[\vhh\right]\Phi\left(\frac{\delta}{2C}\right)}, \numberthis\label{eq:P2 final bound}
    \end{align*}
    where the upper and lower tails of the Poisson distribution are estimated thanks to Bennett's inequality as in \eqref{eq:upper and lower tail Poisson distribution via Bennett}.

    We thus have shown, recalling \eqref{eq:P1 P2 bound} together with the bounds \eqref{eq:P1 final bound part 1}, \eqref{eq:P1 final bound part 2} and \eqref{eq:P2 final bound}, that
    \begin{equation}\label{eq:P1 P2 final bound}
        \P \left[\sup_{x \in \Lambda^N|_{\vh}} \abs{\overline{k}(x) - \E\Big[\overline{k}(x)\Big]} > \delta\right] \le C_5N e^{-C_6 Nh_N^m\delta^2}
    \end{equation}
    for $\delta > 0$ small enough and some constants $C_5, C_6 > 0$ that depend only on $k, m, \mu$ and $\varphi$.

    \indent \textbf{Step 2 (Comparing the discrete and continuous RW)}.

    Conditioned on the event
    \begin{equation*}
        E_{\delta} \coloneqq \left(\sup_{x \in \Lambda^N|_{\vh}} \abs{\overline{k}(x) - \E\Big[\overline{k}(x)\Big]} \le \delta\right),
    \end{equation*}
    the elementary inequality $|a^{-\alpha} - b^{-\alpha}| \le \alpha c^{-(\alpha+1)}|a-b|$ valid for all $a, b > c$ together with the lower bound $C'$ and upper bound $C$ of $\E[\overline{k}(x)]$ on $\vhh$ from \cref{thm:lower bound upper bound E overline k} yield
    \begin{align*}
        \sup_{y \in \Lambda^N|_{\vh}} \abs{\frac{1}{\left(\overline{k}(x)\overline{k}(y)\right)^\alpha} - \frac{1}{\left(\E\Big[\overline{k}(x)\big]\E\Big[\overline{k}(y)\Big]\right)^\alpha}} &\le \alpha\left(C' - \delta\right)^{-(2\alpha + 2)}\abs{\overline{k}(x)\overline{k}(y) - \E\Big[\overline{k}(x)\Big]\E\Big[\overline{k}(y)\Big]} \\
            &\le 2\alpha\left(C'- \delta\right)^{-(2\alpha+2)}\left(C + \delta\right)\delta \le C_7\delta 
    \end{align*}
    for $\delta$ small enough.
    Hence, conditioned on $E_\delta$,
    \begin{equation}\label{eq:overline Z vs overline Z_E estimate}
        \abs{\overline{Z}(x) - \overline{Z_\E}(x)} \le \abs{\overline{k}(x)} \sup_{y \in \Lambda^N|_{\vh}}\abs{\frac{1}{\left(\overline{k}(x)\overline{k}(y)\right)^\alpha} - \frac{1}{\left(\E\Big[\overline{k}(x)\big]\E\Big[\overline{k}(y)\Big]\right)^\alpha}} \norm{\nabla\varphi}_\infty h_N^{-1} \le C_8\delta h_N^{-1},
    \end{equation}
    and similarly,
    \begin{equation}\label{eq:overline Z_E vs expectation overline Z_E estimate}
        \abs{\overline{Z_\E}(x) - \E\Big[\overline{Z_\E}(x)\Big]} \le C_9\delta h_N^{-1}
    \end{equation}
    for constants $C_8, C_9 > 0$ that only depend on $\alpha, k, \mu, m,$ and $ \varphi$.
    By combining estimates \eqref{eq:overline Z vs overline Z_E estimate} and \eqref{eq:overline Z_E vs expectation overline Z_E estimate} with the law of total probability, we obtain
    \begin{align*}
        \P\Bigg[\sup_{x \in \Lambda^N|_{\vh}} & \abs{\overline{Z}(x) - \E\Big[\overline{Z_\E}(x)\Big]} > \varepsilon\Bigg] \\
            &\le \P\left[\sup_{x \in \Lambda^N|_{\vh}} \abs{\overline{Z}(x) - \overline{Z_\E}(x)} + \abs{\overline{Z_\E}(x) - \E\Big[\overline{Z_\E}(x)\Big]} > \varepsilon \Bigg|\, E_\delta \right] + \P\left[E_\delta^C\right] \\
            &\le \Ind_{\delta > C_{10} h_N\varepsilon} + \P\left[E_\delta^C\right] \numberthis\label{eq:overline Z vs expectation overline Z_E first upper bound}
    \end{align*}
    with $C_{10} \coloneqq 1/(C_8 + C_7)$. 
    
    Now, we set $\delta \coloneqq C_{10}h_N\varepsilon$.
    Then the indicator in \eqref{eq:overline Z vs expectation overline Z_E first upper bound} vanishes, and the second term is estimated thanks to \eqref{eq:P1 P2 final bound};
    \begin{equation*}
        \P\left[E_\delta^C\right] \le C_5N
        e^{-C_6C_{10}^2 Nh_N^{m+2}\varepsilon^2}.
    \end{equation*}
    This probability is summable in $N$ in the considered regime $Nh_N^{m+2}/\log N \to +\infty$ with $h_N \to 0$ as $N\to+\infty$. 
    The Borel--Cantelli lemma then yields almost sure convergence which concludes the proof.
\end{proof}
\section{Lifting the SEP to principal bundles}\label{bundlesect}

We extend the construction of RNG and the SEP to a total $\widetilde m$-dimensional manifold $\widetilde M$ lying above the $m$-dimensional manifold $M$ with $\Ric_M \ge (m-1)\kappa$ for some $\kappa \le 0$ (or assume that $M$ is compact).
We work in a bundle setting which allows for a canonical comparison of points in different fibers.
\newline

\noindent \textbf{(Geometric setting)}. Let $G$ be a compact Lie group and $\pi:\widetilde M\to M$ be a principal $G$-bundle with right action $R_a(u)=u\cdot a$ ($u \in \widetilde M$, $a \in G$) and Lie algebra
$\mathfrak g=T_eG$ equipped with
a principal connection $\omega$.
We endow $G$ with a bi-invariant Riemannian metric which yields
a $G$-invariant metric $\widetilde g^\V$ on the vertical bundle $\V \coloneqq \ker(d\pi)$; that is
\begin{equation*}
    \widetilde g^\V(u \cdot a) \left((R_a)_* w, (R_a)_* w'\right) = \widetilde g^\V(u)(w, w'), \quad w, w' \in \V_u,
\end{equation*}
as follows:
Given an Ad-invariant inner product $\langle\cdot,\cdot\rangle_{\mathfrak g}$ on $\mathfrak g$ (equivalently, a bi-invariant Riemannian metric on $G$) and $\xi\in\mathfrak g$, define the
\emph{fundamental vertical vector field}
\begin{equation*}
    \xi^\#(u):=\left.\frac{d}{dt}\right|_{t=0}u\cdot\exp(t\xi)\in \V_u.
\end{equation*}
The map $\xi\mapsto \xi^\#(u)\in \V_u$ is a linear isomorphism $\mathfrak g \simeq \V_u$, hence it induces a metric
$\widetilde g^\V$ on the \emph{vertical bundle} $\V \coloneqq \ker d\pi$ by
\begin{equation*}
    \widetilde g^\V\big(\xi^\#(u),\eta^\#(u)\big) \coloneqq\langle\xi,\eta\rangle_{\mathfrak g},\quad \xi,\eta\in\mathfrak g.
\end{equation*}
By Ad-invariance, $g_\V$ is $G$-invariant, namely, the right action $R_{(\cdot)}$ acts by fiberwise isometries on $\V$. We denote by $d_{\widetilde M_x}$ the distance in the fiber $\widetilde M_x \coloneqq \pi^{-1}(x)$ above $x \in M$ and $m' \coloneqq \dim \V$ so that $\widetilde m = m + m'$.

The principal connection $\omega$ defines the horizontal distribution $\H \coloneqq \ker(\omega)$ and splits the tangent bundle as
\begin{equation*}
T\widetilde M = \H \oplus \V.
\end{equation*}
With this splitting, we define the bundle metric,
\begin{equation}\label{eq:bundle metric KK}
    \widetilde g \coloneqq \pi^* g \oplus \widetilde g^\V.
\end{equation}
If the compact Lie group $G$ is semi-simple, then the (negative of the) killing form $B$ defined by
\begin{equation*}
    \quad B(\xi, \eta) \coloneqq \Tr \left(\mathrm{Ad}(\xi)\mathrm{Ad}(\eta)\right), \quad \xi, \eta \in \mathfrak{g},
\end{equation*}
also yields a bi-invariant metric $-B$ on $\mathfrak g$ \cite[Corollary 2.33]{alexandrino_lie_2015}, and the total manifold $\widetilde\M$ admits the metric
\begin{equation*}\label{eq:Sasaki-Mok type metric}
    \widetilde g \coloneqq \pi^*g \oplus (-B)\omega \coloneqq \widetilde g^\H \oplus \widetilde g^\V.
\end{equation*}
This encompasses, for instance, the orthonormal frame bundle $\pi_{O(\M)} : O(\M) \to \M$ (and $O(m)$ is compact semi-simple for $m \ge 3$)  used in the construction of Riemannian Brownian motion by Eells--Elworthy--Malliavin (see \cite[Section 3.1]{hsu_stochastic_2002}), 

The connection also allows to lift curves and tangent vectors from the base manifold $\M$ to the total manifold $\widetilde\M$ in a coherent way.
\begin{definition}[Horizontal lift]
Given a smooth curve $\gamma\in C^\infty(I;M)$ on an interval $I\ni0$ and $u_0\in\pi^{-1}(\gamma(0))$,
the \emph{horizontal lift} $\widetilde\gamma$ of $\gamma$ with respect to $\H$ starting at $u_0$ is the unique curve satisfying
\begin{equation*}
    \pi \circ \widetilde\gamma = \gamma, \quad \widetilde\gamma '(t) \in \H_{\widetilde \gamma(t)}.
\end{equation*}
The same procedure applies to tangent vectors: given $v\in T_{\pi(u)}M$, its \emph{horizontal lift} at $u\in\widetilde M$ with respect to $\H$
is the unique vector $\tilde v\in \H_u$ such that $d\pi(\tilde v)=v$.
\end{definition}
This yields a canonical transport between fibers: for $x,y\in M$ sufficiently close and
$u\in \widetilde M_x$, we denote by $\widetilde P_{yx} u\in \widetilde M_y$ the endpoint of the horizontal lift of the (unique minimizing) geodesic from $x$ to $y$ starting at $u$. This map $\widetilde P_{yx}: \widetilde M_x \to \widetilde M_y$ is a diffeomorphism between the fibers. Because the fiber metric is $G$-invariant, $\widetilde P_{yx}$
acts by fiber isometries,
\begin{equation*}
    d_{\widetilde M_y}(\widetilde P_{yx} u_1, \widetilde P_{yx} u_2) = d_{\widetilde M_x}(u_1, u_2), \quad u_1, u_2 \in \widetilde M_x,
\end{equation*}
and in particular, for $u \in \widetilde M_x,\, q \in \widetilde M_y$,
\begin{equation}\label{eq:fibre-isometry}
d_{\widetilde M_y}(\widetilde P_{yx} u,q) = d_{\widetilde M_x}(\widetilde P_{xy}\widetilde P_{yx} u, \widetilde P_{xy} q) = d_{\widetilde M_x}(u, \widetilde P_{xy} q).
\quad 
\end{equation}

\begin{definition}[Totally geodesic fibers]
The fibers of a Riemannian submersion are said to be totally geodesic if any geodesic in a fiber (for the induced metric)
is also a geodesic in $\widetilde M$.
\end{definition}

In our setting, the fibers are totally geodesic for the Riemannian submersion $\pi : \widetilde M \to M$ with the Kaluza--Klein bundle metric \eqref{eq:bundle metric KK} (see e.g. \cite[Theorem 1]{nagy1985geodesies}).
For totally geodesic fibers, the horizontal Laplacian coincides with the Bochner horizontal Laplacian in sub-Riemannian geometry with the decomposition $\Delta_{\widetilde\M} := \Delta^\H + \Delta^\V$ into horizontal and vertical parts \cite[Section 4.1]{baudoin_stochastic_2023}. We denote by $X^\H$ and $X^\V$ the horizontal and vertical parts of a vector field $X \in \Gamma(T\widetilde M)$ and by $\widetilde \nabla$ the Levi-Civita connection on $\widetilde M$.
\begin{definition}[Horizontal Laplacian]\label{def:Rough horizontal Laplacian}
    Given a smooth function $\phi \in C^\infty(\widetilde\M)$, the \emph{horizontal Laplacian} $\Delta^\H$ is the trace of the \emph{horizontal Hessian}; i.e., 
    \begin{equation*}
        \Delta^\H\phi \coloneqq \Tr\left(\widetilde g^{-1} \otimes \Hess^\H\phi\right) \coloneqq \widetilde g^{ij} \Hess^\H_{ij}\phi,
    \end{equation*}
    where the $\widetilde g^{ij}$'s are the components of the inverse metric on $\widetilde M$, and the \emph{horizontal part of the Hessian} is the bilinear form given by
    \begin{equation*}
        \forall X, Y \in \Gamma(T\widetilde\M),\quad \Hess^\H\phi(X, Y) \coloneqq \left(\widetilde\nabla_X \left(\widetilde\nabla\phi\right)^\H\right) \cdot Y^\H.
    \end{equation*}
\end{definition}
By construction, since the distributions $\H$ and $\V$ are orthogonal to each other, the trace of the horizontal Hessian can be taken with respect to the inverse of the horizontal metric $\widetilde g^\H \coloneqq \pi^* g$. In local orthonormal frames $(Z_i)_{1\le i\le m}$ of $\H$ and $(F_j)_{1\le j \le m'}$ of $\V$, the horizontal Laplacian can be rewritten as
\begin{equation*}\label{eq:rough Laplacian}
        \Delta^\H = \sum_{i=1}^{m} \left(Z_i^2 - \left(\widetilde\nabla_{Z_i} Z_i\right)^\H\right) - \sum_{j=1}^{m'} \left(\widetilde\nabla_{F_j} F_j\right)^\H.
    \end{equation*}
In our setting, since the fibers are totally geodesic, the last term vanishes;
\begin{equation*}
    \sum_{j=1}^{m'} \left(\widetilde\nabla_{F_j} F_j\right)^\H = 0.
\end{equation*}

\noindent \textbf{(The lifted SEP)}. 
Random neighborhood graphs on the base manifold $\M$ can be lifted in a consistent manner to the total manifold $\widetilde\M$ via the submersion.
\begin{definition}[Lifted random neighbourhood graph (lifted RNG)]\label{def:Lifted random neighbourhood graph}
    Conditionally on the $\mathrm{PPP}$ $\Lambda_N = \sum_i \delta_{X^i}$ of intensity $N d\mu$ in \cref{def:Random neighbourhood graph}, for each $i$, let $(\Lambda')^{N'}_i$ be an independent $\mathrm{PPP}$ on the fiber $\widetilde M_{X^i}$ with intensity $N' d\Vol_{\widetilde M_{X^i}}$; the fiber volume is induced by $g^\V$. 
    The \emph{vertices} $V_{N,N'}$ consist of the basepoints $X^i$ of $\Lambda^N$ and their fiber counterparts $X^i_a$ of $(\Lambda')_i^{N'}$. 
    An \emph{undirected edge} is present between two vertices whenever they are close on the base manifold and in the fibers; i.e., $d_M(X^i, X^j) \le h_N$ and $d_{\widetilde M_{X^j}}(\widetilde P_{X^jX^i} X^i_a, X^j_b)$ for some \emph{bandwidth parameters} $h_N, h_{N'}$. 
    The sequence of \emph{lifted random neighborhood graphs} \emph{(}lifted \emph{RNG)} is given by $(G_{N,N'} \coloneqq (V_{N,N'}, E_{N,N'}))_{N,N' \ge 1}$, where $E_{N,N'}$ denotes the set of undirected edges.
\end{definition}
Following Gao \cite{gao_diffusion_2021}, we consider weights mimicking the ones of the SEP on the base manifold $\M$ given by \eqref{eq:symmetric weights for thm}.
We take into account the additional geometrical information of the total manifold $\widetilde\M$.
Given a bounded measurable function $k_{\widetilde\M} : \R^+ \times \R^+ \to \R^+$ with compact support $[0,1] \times [0,1]$, that we refer as a \emph{kernel}, we set the following \emph{weights},
\begin{equation}\label{eq:symmetric for theorem lifted}
    \widetilde W_{N,N'}((X^i, X^i_a), (X^j, X^j_b)) \coloneqq \frac{h_N^{-m}(h'_{N'})^{-m'} k_{\widetilde\M}\left(\frac{d_\M(X^i, X^j)}{h_N}, \frac{d_{\widetilde\M_{X^j}}(\widetilde{\mathrm{P}}_{X^jX^i} X^i_a,\, X^j_b)}{h'_{N'}}\right)}{\sqrt{e^{-U(X^i)}e^{-U(X^j)}}}.
\end{equation}
For the SEP on the sequence $(G_{N,N'})_{N,N' \ge 1}$ of lifted RNG, we consider the weights given by \eqref{eq:symmetric for theorem lifted}. These weights are in fact \emph{symmetric} in light of \eqref{eq:fibre-isometry} due to the fiber isometry property. The existence of the SEP follows as in \cref{sec:Existence}.

\begin{remark}
    The two bandwidth parameters $h_N$ and $h'_{N'}$ tune the horizontal and vertical parts respectively. In the \emph{horizontal diffusive scaling}, $h_N'$ goes to zero faster than $h_N$ when both $(N,N') \to +\infty$ in order to thermalize in the fiber while having diffusion only along the base manifold.
\end{remark}

\begin{theorem}[Hydrodynamic limit of the lifted SEP]\label{Hydrodynamic limit thm}
    Assume that either $\Ric_M \ge (m-1)\kappa$ for some $\kappa \le 0$ or that $M$ is compact. Let $(\widetilde G_{N,N'} = (\widetilde V_{N,N'}, \widetilde E_{N,N'}))_{N,N' \ge 1}$ be a sequence of lifted $\mathrm{RNG}$ on $\widetilde\M$ in the sense of \cref{def:Lifted random neighbourhood graph} with symmetric weights $(\widetilde W_{N,N'})_{N,N'\ge1}$ given by \eqref{eq:symmetric for theorem lifted}. Let $\widetilde\eta_0$ be an initial configuration of the \emph{SEP} on $G_{N,N'}$ with associated initial empirical measure $\widehat{\pi_0}^{N,N'}$ that converges weakly to $\widetilde\rho_0 e^{-\widetilde U} d\Vol_{\widetilde\M}$ in probability; that is,
    \begin{equation*}
        \forall \varphi \in C_0(\widetilde\M), \quad \left\langle \varphi, d\widehat{\pi_0}^{N,N'} \right\rangle \overset{\mathrm{pr.}}{\to} \left\langle \varphi, \widetilde\rho_0 e^{-\widetilde U} \; d\Vol_{\widetilde\M}\right\rangle \quad \text{as }(N,N')\to +\infty,
    \end{equation*}
    where $\widetilde U \coloneqq U \circ \pi$ is the lift of $U$ and
    \begin{equation*}
         \quad \widehat{\pi_0}^{N,N'} \coloneqq \frac{1}{N}\frac{1}{N'} \sum_{(X^i, X^i_a) \in V_{N,N'}} \eta_0(X^i, X^i_a) \delta_{X^i,X^i_a}.
    \end{equation*}
    Suppose that the bandwidth parameters $(h_N, h'_{N'})_{N,N'\ge1}$ are chosen in such a way that both $h_N \to 0$, $h'_{N'} \to 0$ with $h'_{N'}/h_{N} \to 0$ and $NN'h_N^{m+2}(h'_{N'})^{m'}/\log (NN') \to +\infty$ as $(N, N') \to +\infty$. Then for almost all realizations of the sequence of graphs, the trajectory of the empirical measure $t\mapsto \widehat{\pi_t}^{N,N'}$ associated to the configuration $\widetilde \eta_t^{N,N'}$ that evolves according to the rescaled $\mathrm{SEP}$ with generator $h_N^{-2}L_{\widetilde G_{N,N'}}^{\SEP}$ converges in the Skohorod topology to a deterministic trajectory $t\mapsto \widetilde\rho_t e^{-\widetilde{U}}d\Vol_{\widetilde\M}$ whose density is the weak solution of
    \begin{equation*}\label{eq:horizontal heat drift}
        \partial_t \widetilde\rho_t = C_{k_{\widetilde\M}}\left(\Delta^\H - \widetilde\nabla \widetilde U \cdot \widetilde \nabla\right)\widetilde\rho_t,
    \end{equation*}
    with initial condition $\widetilde\rho_0$ and for some constant $C_{k_{\widetilde\M}}$ that depends only on moments of the kernel $k_{\widetilde\M}$ in \eqref{eq:symmetric for theorem lifted}, provided there exists a unique weak solution. 
\end{theorem}
\begin{remark}
    Note that when the initial profile $\widetilde\rho_0$ on the total manifold $\widetilde\M$ is $\pi$-related to some initial profile $\rho_0$ on the base manifold $\M$ ($\widetilde\rho_0 = \rho_0 \circ \pi$), then the evolution of the latter is given by
    \begin{equation*}
        \partial_t \rho_t = C_{k,\H} \left(\Delta_\M - \nabla U \cdot \nabla\right)\rho_t.
    \end{equation*}
This follows by the $\pi$-relation between the horizontal Laplacian and the Laplacian of the underlying base manifold; for all smooth functions $\phi \in C^\infty(\M)$, it holds
\begin{equation*}
    \Delta^\H (\phi\circ\pi) = \left(\Delta_\M \phi\right) \circ \pi.
\end{equation*}
This identity can be found, for instance, in the monographs of Hsu \cite[Proposition 3.1.2]{hsu_stochastic_2002} for the orthonormal frame bundle and of Baudoin, Demni and Wang \cite[Proposition 4.2.7]{baudoin_stochastic_2023} for submersions with totally geodesic fibers. For a probabilistic proof, we refer to \cite[Proposition 3.13]{junne_invariance_2026}.
\end{remark}
Again, the hydrodynamic limit is a consequence of duality between the SEP and RW. 
Gao obtained the almost-sure pointwise consistency of (normalized) graphs Laplacians for neighborhood graphs lying on general fibre bundles \cite[Theorems 2 \& 4]{gao_diffusion_2021} and obtained rates for the unit tangent bundle $UT\M \to \M$ \cite[Theorem 9]{gao_diffusion_2021}. The weights are of the form
\begin{equation}\label{eq:kernel lifted symmetric weights}
    \widetilde W_{N,N';\alpha}((X^i, X^i_a), (X^j, X^j_b)) \coloneqq \frac{h_N^{-m}(h'_{N'})^{-m'} k_{\widetilde\M}\left(\frac{d_\M(X^i, X^j)}{h_N}, \frac{d_{\widetilde\M_{X^j}}(\widetilde{\mathrm{P}}_{X^jX^i} X^i_a, X^j_b)}{h'_{N'}}\right)}{\left(\overline{k_{\widetilde\M}}(X^i, X^i_a) \overline{k_{\widetilde\M}}(X^j, X^j_b)\right)^\alpha}
\end{equation}
with
\begin{equation*}
    \overline{k_{\widetilde\M}}(X^i, X^i_a) \coloneqq \frac{1}{N}\frac{1}{N'}\sum_{(X^j, X^j_b) \in (\Lambda^N, (\Lambda')^{N'})} h_N^{-m}(h'_{N'})^{-m'} k_{\widetilde\M}\left(\frac{d_\M(X^i, X^j)}{h_N}, \frac{d_{\widetilde\M_{X^j}}(\widetilde{\mathrm{P}}_{X^jX^i} X^i_a, X^j_b)}{h'_{N'}}\right).
\end{equation*}
\begin{definition}[Lifted Weighted horizontal Laplacian]\label{def:Lifted weighted horizontal Laplacian}
    Given a smooth  function $\phi \in C^\infty(\widetilde\M)$, the \emph{lifted $\alpha$-weighted horizontal Laplacian} with respect to the Gibbs reference measure $d\mu = e^{-U} d\Vol_\M$ on $\M$ is the reversible operator with respect to $e^{-\alpha\widetilde U} \:d\Vol_{\widetilde\M}$; that is, the operator satisfying
    \begin{equation*}
        \forall \varphi \in C_c^\infty(\widetilde\M), \quad \int_{\widetilde\M} \varphi \Delta_{\alpha}^\H \phi \:e^{-\alpha\widetilde U}d\Vol_{\widetilde\M} = -\int_{\widetilde \M} \left(\widetilde \nabla \varphi\right)^\H \cdot \left(\widetilde \nabla \phi\right)^\H \:e^{-\alpha\widetilde U}d\Vol_{\widetilde\M},
    \end{equation*}
    which is given by
    \begin{equation*}\label{eq:Weighted horizontal Laplacian Delta_alpha}
        \Delta_\alpha^\H \coloneqq e^{\alpha\widetilde U} \widetilde\nabla \cdot \left(e^{-\alpha\widetilde U} \widetilde\nabla^\H \right) = \Delta^\H - \alpha\widetilde \nabla \widetilde U \cdot \widetilde \nabla.
    \end{equation*}
\end{definition}
Cast in our setting, following the line of proof of \cref{thm:almost sure convergence empirical URW to Delta 2(1-alpha) general}, we obtain the desired convergence.
\begin{theorem}\label{thm:almost sure convergence empirical lifted URW to horizontal Delta 2(1-alpha) general}
    Let $\varphi \in C_c^\infty(\widetilde\M)$ and let $(\widetilde G_{N, N'})_{N,N' \ge 1}$ be a sequence of lifted $\mathrm{RNG}$ in the sense of \cref{def:Lifted random neighbourhood graph} with symmetric weights $(\widetilde W_{N,N';\alpha})_{N,N' \ge 1}$ given by \eqref{eq:kernel lifted symmetric weights} and associated RW with rescaled generator $h_N^{-2}L_{\widetilde G_{N,N'}}^{\RW}$. There are constants $\widetilde C_{\alpha;1,2,3} > 0$ that only depend on $\alpha, k_{\widetilde\M}, \varphi$, $\mu$, $m$ and on $m'$, and a constant $C_{k_{\widetilde\M};\alpha} > 0$ that depends on the moments of $k_{\widetilde\M}$ and $\alpha$ such that, for all $\varepsilon > 0$ and $h_N, h'_{N'}$ small enough,
    \begin{align*}
        \P\Bigg[\sup_{(X^i, X^i_a) \in V_{N,N'}}\Bigg| \Bigg(h_N^{-2} L_{\widetilde G_{N,N'}}^{\RW}\varphi - C_{k_{\widetilde\M};\alpha}&\left(e^{(2\alpha-1)\widetilde U}\Delta^\H_{2(1-\alpha)}\varphi\right)\!\!\Bigg)(X^i, X^i_a)
        + O\Big(h_N + h_N^{-2}(h'_{N'})^{2}\Big)\Bigg| > \varepsilon\Bigg]
            \\ &\le \widetilde C_{\alpha;1}NN' \exp\left(-\widetilde C_{\alpha;2} Nh_N^{m}(h'_{N'})^{m'} \tfrac{h^4_N}{(h_N + h'_{N'})^2}\varepsilon^2\right).
    \end{align*}
    Then if both $h_N \to 0$, $h'_{N'} \to 0$ with $h'_{N'}/h_{N} \to 0$ and $NN' h_N^{m+2}(h'_{N'})^{m'}/\log (NN') \to +\infty$ as $(N, N') \to +\infty$, it holds almost surely and uniformly in the vertices $(X^i, X^i_a)$'s,
    \begin{equation*}
        \lim_{(N,N') \to +\infty} h_N^{-2} L_{N,N';\,s}^{\URW}\varphi(X^i, X^i_a) = C_{k_{\widetilde\M};\alpha} e^{(2\alpha-1)\widetilde U(X^i, X^i_a)}\Delta^\H_{2(1-\alpha)}\varphi(X^i, X^i_a).
    \end{equation*}
\end{theorem}

\appendix
\section{Computational lemmas}\label{sec:Computational lemmas}
In this appendix we gather several computational lemmas characterizing the limiting operator associated to the rescaled RW. The limits of the integral operators can be found in Hein, Audibert, von Luxurg \cite[Proposition 22 \& Theorem 25]{hein_graph_2007} and Gao \cite[Lemmas 13,14,15 \& Theorem 4] {gao_diffusion_2021} among others.
\begin{lemma}\label{thm:1st 2nd 3rd moments k}
    Denote by $S_{m-1} \coloneqq 2\pi^{m/2}/\Gamma(m/2)$ the surface of the $(m-1)$-dimensional unit sphere in $\R^m$ and by $B_m(0,1)$ its unit ball. The first three moments of $k$ are given by
    \begin{align*}
        M^k_0 &\coloneqq \int_{B_m(0,\,1)} k\left(\abs{v}\right) \: dv = S_{m-1} \int_0^1 k(r) r^{m-1} \: dr \coloneqq C_{k,0}, \\
        M^k_1 &\coloneqq \int_{B_m(0,\,1)} k\left(\abs{v}\right) v^i \: dv = 0, \\
        M^k_2 &\coloneqq \int_{B_m(0,\,1)} k\left(\abs{v}\right) v^iv^j \: dv = \left(\frac{S_{m-1}}{m} \int_0^1 k(r) r^{m+1} \: dr\right) \delta^{ij} \coloneqq C_{k,2} \delta^{ij}, \\
        M^k_3 &\coloneqq \int_{B_m(0,\,1)} k\left(\abs{v}\right) v^iv^jv^l \: dv = 0.
    \end{align*}
\end{lemma}
\begin{proof}
    Regarding the odds moments, the reflection $v \mapsto -v$ changes the sign of the integrated quantity; its integral must be zero. For the second moment, reflecting a single coordinate, we deduce that $\int k(\abs{v})v^iv^j \: dv \propto \delta^{ij}$. The proportionality constant $C_{k,2}$ is then obtained noting that
    \begin{equation*}
        \Tr\left(C_{k,2}(\delta^{ij})_{ij}\right) = C_{k,2}m = \sum_{i=1}^m \int_{B_m(0,\, 1)} k\left(\abs{v}\right) \left(v^i\right)^2 \: dv = S_{m-1}\int_0^1 k(r)r^{m+1} \: dr.
    \end{equation*}
\end{proof}
\begin{lemma}\label{thm:lower bound upper bound E overline k}
    Let $\varphi \in C_c^\infty(\M)$ be a test function. Then uniformly in $x \in \vh$,
    \begin{equation}\label{eq:expectation overline k}
        \int_\M h^{-m} k\left(\frac{d_\M(x, y)}{h}\right) \: d\mu(y) = C_{k,0} e^{-U(x)} + O(h^2).
    \end{equation}
    In particular, there are two constants $C, C' > 0$ such that
    \begin{equation*}
        \forall x \in (\Supp \varphi)_{h_1},\quad C' \le \int_\M h^{-m}k\left(\frac{d_\M(x, y)}{h}\right) \: d\mu(y) \le C.
    \end{equation*}
\end{lemma}
\begin{proof}
    Going to normal coordinates about $x$ yields
    \begin{align*}
        \int_\M h^{-m} k\left(\frac{d_\M(x, y)}{h}\right) \: d\mu(y) &= e^{-U(x)}\int_{B_{T_x\M}(0,1)} k\left(\abs{v}\right) e^{-\left(U(\exp_x(hv)) - U(x)\right)} \sqrt{\det g(\exp_x(hv))} \:dv \\
            &= e^{-U(x)}\int_{B_{m}(0,1)} k\left(\abs{v}\right) \left(1 - E_iUv^ih + O(h^2)\right) \:dv \\
            &= C_{k,0} e^{-U(x)} + O(h^2),
    \end{align*}
    where we used in the penultimate line that the odd moment vanishes by \cref{thm:1st 2nd 3rd moments k}.
\end{proof}
\begin{lemma}\label{thm:convergence integral to weighted Laplacian}
    Let $\varphi \in C_c^\infty(\M)$ be a test function. Then uniformly in $x \in \M$,
        \begin{equation}\label{eq:limit laplacian URW 2(alpha-1)}
            h^{-2} \int_{\M} \frac{h^{-m} k\left(\frac{d_\M(x, y)}{h}\right) \left(\varphi(y) - \varphi(x)\right)}{\left(\E\Big[\overline{k}(x)\Big]\right)^\alpha\left(\E\Big[\overline{k}(y)\Big]\right)^\alpha} \: d\mu(y) = \frac{C_{k,2}}{2\left(C_{k,0}\right)^{2\alpha}}e^{(2\alpha-1)U(x)}\Delta_{2(1-\alpha)}\varphi(x) + O(h^2).
        \end{equation}
\end{lemma}
\begin{proof}
    We expand the denominator; $\Lambda^N$ is a ($\mathrm{PPP}$) of intensity $N d\mu$, hence
    \begin{equation}\label{eq:expectation def overline k}
        \E\Big[\overline{k}(x)\Big] = \E\left[\frac{1}{N} \sum_{X^l \in \Lambda^N} h^{-m}k\left(\frac{d_\M(x, X^l)}{h}\right)\right] = \int_\M h^{-m} k\left(\frac{d(x,z)}{h}\right) \: d\mu(z).
    \end{equation}
    Substituting the expectation \eqref{eq:expectation def overline k} by \eqref{eq:expectation overline k} in the denominator to get
    \begin{align*}
    h^{-2} \int_{\M} &\frac{h^{-m} k\left(\frac{d_\M(x, y)}{h}\right) \left(\varphi(y) - \varphi(x)\right)}{\left(\E\Big[\overline{k}(x)\Big]\right)^\alpha\left(\E\Big[\overline{k}(y)\Big]\right)^\alpha} \: d\mu(y) \\
            = &h^{-2}\frac{e^{(2\alpha-1)U(x)}}{\left(C_{k,0}\right)^{2\alpha}}\int_{B_{T_x\M}(0,1)} k\left(\abs{v}\right) \left(\varphi\left(\exp_x(hv)\right) - \varphi(x)\right) e^{-(1-\alpha)\left(U(\exp_x(hv)) - U(x)\right)} \\
            &\quad \times \sqrt{\det g(\exp_x(hv))}\left(1 + O(h^2)\right) \:dv
    \end{align*}
    Expanding the right-hand side yields
    \begin{align*}
            &h^{-2}\frac{e^{(2\alpha-1)U(x)}}{\left(C_{k,0}\right)^{2\alpha}}\int_{B_{T_x\M}(0,1)} k\left(\abs{v}\right) \Bigg\{\left(d_x\varphi(v)h + \frac{1}{2}\Hess_x \varphi(v,v)h^2 + \frac{1}{6}\nabla\Hess_x\varphi(v,v,v)h^3\right) \\
            &\quad \Big(1 + (\alpha-1)d_xU(v)h + O(h^2)\Big)\left(1 - \frac{1}{6}\Ric_x(v,v)h^2\right)\left(1 + O(h^2)\right)\Bigg\} + O(h^4)\: dv \\
            & = h^{-2}\frac{e^{(2\alpha-1)U(x)}}{\left(C_{k,0}\right)^{2\alpha}}\int_{B_{m}(0,1)} k\left(\abs{v}\right) \Bigg\{E_i\varphi v^i h\left(1 + O(h^2)\right) + \left(\frac{1}{2}E_iE_j\varphi + (\alpha-1) E_iU E_j\varphi\right)v^iv^jh^2 \\ 
            &\quad + \left(\frac{1}{6}\left(E_iE_jE_l\varphi - \Ric_{ij}E_l\varphi\right) + \frac{\alpha-1}{2}E_iUE_jE_l\varphi\right)v^iv^jv^lh^3 \Bigg\} + O(h^4) \: dv.
    \end{align*}
    The odds moments vanish by \cref{thm:1st 2nd 3rd moments k} while the second moment recovers the weighted Laplacian;
    \begin{align*}
        \left(\frac{1}{2}E_iE_j\varphi + (\alpha-1) E_iU E_j\varphi\right)\int_{B_m(0;\,1)} k\left(\abs{v}\right) v^iv^j \: dv
            &= \frac{C_{k,2}}{2}\left(\Delta - 2(1-\alpha)\nabla U \cdot \nabla\right)\varphi \\
            &\coloneqq \frac{C_{k,2}}{2}\Delta_{2(1-\alpha)}\varphi.
    \end{align*}
\end{proof}
\begin{lemma}\label{thm:1st 2nd 3rd moments k_h}
    The first moments of $k_{\widetilde\M}$ are given by
    \begin{align*}
        &M^{k_{\widetilde\M}}_0 \coloneqq \int_{B_m(0;\,1)}\int_{B_{m'}(0; \,1)} k_{\widetilde\M}\left(\abs{v},  \abs{w}\right) \: dw dv \\ 
        &\hphantom{M^{k_{\widetilde\M}}_0 \:} = S_{m-1}S_{m' - 1} \int_0^1\int_0^1 k_{\widetilde\M}(r,r') r^{m-1}(r')^{m' - 1} \: dr'dr \coloneqq C_{k_{\widetilde\M}, 0}, \\
        &M^{k_{\widetilde\M}}_{\cdot,1} \coloneqq \int_{B_{m'}(0,1)} k_{\widetilde\M}\left(\frac{d_\M(x,y)}{h},  \abs{w}\right) w^j \: dw = 0, \\
        &M^{k_{\widetilde\M}}_{1,0} \coloneqq \int_{B_m(0,1)}\int_{B_{m'}(0,1)} k_{\widetilde\M}\left(\abs{v},  \abs{w}\right) v^i \: dw dv = 0, \\
        &M^{k_{\widetilde\M}}_{2,0} \coloneqq \int_{B_m(0,1)}\int_{B_{m'}(0,1)} k_{\widetilde\M}\left(\abs{v},  \abs{w}\right) v^iv^j \: dw dv, \\
        &\hphantom{M^{k_{\widetilde\M}}_{2} \:}= \left(\frac{S_{m-1}}{m}S_{m'-1} \int_0^1\int_0^1 k_{\widetilde\M}(r,r') r^{m+1}(r')^{m'-1} \: dr'dr\right) \delta^{ij} \coloneqq C_{k_{\widetilde\M},2}\delta^{ij}.
    \end{align*}
\end{lemma}
\begin{proof}
    The proof is similar to the one of \cref{thm:1st 2nd 3rd moments k}.
\end{proof}
\begin{lemma}\label{thm:denom expectation lifted}
    Let $\varphi \in C_c^\infty(\widetilde\M)$ be a test function. Then uniformly in $(x,u)$ on the $(h_1,h_1')$-thickening of the support of $\varphi$ given by
    \begin{equation*}
        (\Supp\varphi)_{h_1,h'_1} \coloneqq \left\{(z,p) \in \widetilde\M; \: d_\M(z, \pi \circ \Supp\varphi) \le h_1, d_{\widetilde\M_z}\left(p, \Supp\varphi_z\right) \le h_1'\right\},
    \end{equation*}
    it holds
    \begin{equation*}
        \int_{\widetilde\M} h^{-m}(h')^{-m'} k_{\widetilde\M}\left(\frac{d_\M(x,y)}{h}, \frac{d_{\widetilde\M_y}(\widetilde P_{yx}u,q)}{h'}\right) e^{-\widetilde U(y,q)}\:d\Vol_{\widetilde\M}(y,q) = C_{k_{\widetilde\M},0}e^{-\widetilde U(x,u)} + O(h^2) + O((h')^2).
    \end{equation*}
\end{lemma}
\begin{proof}
    For $h$ and $h'$ small enough, the restriction of both exponential maps, namely, $\exp_x : B_{T_x\M}(0, h) \simeq B_m(0, h') \to \M$ on the base manifold and $\exp_y^{\widetilde\M_y} : B_{T_{(y,\widetilde{\mathrm{P}}_{yx}u)}\widetilde\M_y}(0, h) \simeq B_{m'}(0, h') \to \widetilde\M_y$ on the fiber are well-defined diffeomorphisms, so that
    \begin{align*}
        \int_\M\int_{\widetilde M_z} h^{-m}&(h')^{-m'}  k_{\widetilde\M}\left(\frac{d_\M(x, y)}{h}, \frac{d_{\widetilde\M_y}\left(\widetilde{\mathrm{P}}_{yx}u, q\right)}{h'}\right) \: d\Vol_{\widetilde\M_y}(q)d\mu(y) \\
            &= \int_{B_{\M}(0,h)}\int_{B_{T_{(y,\widetilde{\mathrm{P}}_{yx}u)}\widetilde\M_y}(0,1)} h^{-m} k_{\widetilde\M}\left(\frac{d_\M(x,y)}{h}, w\right) \sqrt{\det \widetilde g^\V(\exp_{(y,\widetilde{\mathrm{P}}_{yx}u)}^{\widetilde\M_z}(h'w))} \:dw d\mu(y) \numberthis\label{eq:after normal coordiinates fibre}
    \end{align*}
    by going to normal coordinates about $(y, \widetilde{\mathrm{P}}_{yx}u)$. As for the base manifold, the expansion of the fiber's metric is given by
    \begin{equation*}
        \sqrt{\det \widetilde g^\V\left(\exp_{(y,\widetilde{\mathrm{P}}_{yx}u)}^{\widetilde\M_y}(h'w)\right)} = 1 -\frac{1}{6}\Ric_{\widetilde\M}^\V(h'w, h'w) + O((h')^3) = 1 + O((h')^2)
    \end{equation*}
    with an error uniform in both $(x,u)$ and $(y, \widetilde{\mathrm{P}}_{yx}u)$ by compactness of $(\Supp\varphi)_{h_1,h'_1}$. We further expand on the base manifold by going to normal coordinates about $x$ and \eqref{eq:after normal coordiinates fibre} becomes
    \begin{align*}
        \int_{B_{\M}(0;\,h)}\int_{B_{m'}(0;\,1)} &h^{-m} k_{\widetilde\M}\left(\frac{d_\M(x,y)}{h}, w\right)\left(1 + O((h')^2\right)\:dw d\mu(y)  \\
            &= e^{-U(x)} \int_{B_m(0;\,1)}\int_{B_{m'}(0;\, 1)} k_{\widetilde\M}\left(\abs{v}, \abs{w}\right) \left(1 - E_iU v^i h + O(h^2)\right)\left(1 + O((h')^2\right) dwdv \\
            &= C_{k_{\widetilde\M},0}e^{-\widetilde U(x,u)} + O(h^2) + O((h')^2)
    \end{align*}
    by \cref{thm:1st 2nd 3rd moments k_h}.
\end{proof}
\begin{lemma}
    Let $\varphi \in C_c^\infty(\widetilde\M)$ be a test function. Then uniformly in $(x, u) \in \widetilde\M$,
    \begin{multline*}
        h^{-2} \int_{\widetilde\M}\frac{h^{-m} (h')^{-m'} k_{\widetilde\M} \left(\frac{d_\M(x, y)}{h}, \frac{d_{{\widetilde\M}_{y}}(\widetilde{\mathrm{P}}_{yx} u, q)}{h'}\right) \left(\varphi(y,q) - \varphi(x,u)\right)}{\left(\E\Big[\overline{k_{\widetilde\M}}(x,u)\Big]\right)^\alpha\left(\E\Big[\overline{k_{\widetilde\M}}(y,q)\Big]\right)^\alpha}e^{-\widetilde{U}(y,q)} \: d\Vol_{\widetilde\M}(y,q) \\
            = \frac{C_{k_{\widetilde\M},2}}{2\left(C_{k_{\widetilde\M},0}\right)^{2\alpha}} e^{(2\alpha-1)\widetilde U(x, u)}\Delta_{2(1-\alpha)}^\H \varphi(x,u) + O\left(h^{-2}(h')^{2}\right).
    \end{multline*}
\end{lemma}
\begin{proof}
    We decompose the difference of test functions with respect to the geometric structure;
    \begin{equation*}
        \varphi(y,q) - \varphi(x,u) = \left(\varphi(y,q) - \varphi\left(y, \widetilde{\mathrm{P}}_{yx}u\right)\right) + \left(\varphi\left(y, \widetilde{\mathrm{P}}_{yx}u\right) - \varphi(x,u)\right).
    \end{equation*}
    We expand the denominator; $(\Lambda^N, (\Lambda')^{N'})$ is a ($\mathrm{PPP}$) of intensity $NN'e^{-\widetilde U}d\Vol_{\widetilde\M}$
    \begin{align*}
        \E\Big[\overline{k_{\widetilde\M}}(x,u)\Big] &\coloneqq \E\left[\frac{1}{N} \frac{1}{N'}\sum_{X^j \in \Lambda^N}\sum_{X^j_b \in (\Lambda')^{N'}} h^{-m} (h')^{-m'}k_{\widetilde\M}\left(\frac{d_{\widetilde\M}(x, X^j)}{h}, \frac{d_{\widetilde\M_{X^j}}(\widetilde{\mathrm{P}}_{X^j x}u, X^j_b)}{h'}\right)\right] \\
            &= \int_{\M}\int_{\widetilde\M_z} h^{-m} (h')^{-m'} k_{\widetilde\M} \left(\frac{d_\M(x, z)}{h}, \frac{d_{{\widetilde\M}_{z}}(\widetilde{\mathrm{P}}_{zx} u, q)}{h'}\right) \: d\Vol_{\widetilde\M_z}(q)d\mu(z). \numberthis\label{eq:expectation overline def tilde k}
    \end{align*}
    \indent\textbf{Step 1 (Expansion in the fiber)}. We expand the first term of the splitting in normal coordinates about $(y, \widetilde{\mathrm{P}}_{yx}u)$. 
    We replace the denominator by \eqref{eq:expectation overline def tilde k}, which is computed in \cref{thm:denom expectation lifted}, to obtain
    \begin{align*}
        &h^{-2}\int_{\widetilde\M} \frac{h^{-m} (h')^{-m'} k_{\widetilde\M} \left(\frac{d_\M(x, y)}{h}, \frac{d_{{\widetilde\M}_{x}}(\widetilde{\mathrm{P}}_{yx} u, q)}{h'}\right) \left(\varphi(y,q) - \varphi\left(y, \widetilde{\mathrm{P}}_{yx}u\right)\right)}{\left(\E\Big[\overline{k_{\widetilde\M}}(x,u)\Big]\right)^\alpha\left(\E\Big[\overline{k_{\widetilde\M}}(y,q)\Big]\right)^\alpha}e^{-\widetilde U(y,q)} \: d\Vol_{\widetilde\M}(y,q) \\
            &= h^{-2}\frac{e^{\alpha U(x)}}{\left(C_{k_{\widetilde\M},0}\right)^{2\alpha}} h^{-2}\int_{B_\M(0,h)}\int_{B_{T_{(y,\widetilde{\mathrm{P}}_{yx}u)}\widetilde\M_y}(0,1)} h^{-m} k_{\widetilde\M} \left(\frac{d_\M(x, y)}{h}, \abs{w}\right)\left(1 + O(h^2) + O((h')^2)\right) \\
            &\quad \times \left(\varphi\left(\exp_{(y, \widetilde{\mathrm{P}}_{yx}u)}^{\widetilde\M_y}(h'w)\right) - \varphi(y,\widetilde{\mathrm{P}}_{yx}u)\right)e^{-(1-\alpha)U(y)}\sqrt{\det \widetilde g^\V\left(\exp_{(y,\widetilde{\mathrm{P}}_{yx}u)}^{\widetilde\M_y}(h'w)\right)} \: dwd\Vol_\M(y) \numberthis\label{eq:lifted intermediate step 1st term splitting}
    \end{align*}
    Denote by $(F_j)_{1\le j\le m'}$ the orthonormal base of $T_{(y, \widetilde{\mathrm{P}}_{yx}u)}\widetilde\M_y = \V_{(y, \widetilde{\mathrm{P}}_{yx}u)}$ so that
    \begin{equation*}
        \left(\varphi\left(\exp_{(y, \widetilde{\mathrm{P}}_{yx}u)}^{\widetilde\M_y}(h'w_y)\right) - \varphi(y,\widetilde{\mathrm{P}}_{yx}u)\right) = d_{(y,\widetilde{\mathrm{P}}_{yx}u)}\varphi(w)h' + O((h')^2) = F_j\varphi w^j h' + O((h')^2).
    \end{equation*}
    and \eqref{eq:lifted intermediate step 1st term splitting} becomes
    \begin{equation*}
        h^{-2}\frac{e^{\alpha U(x)}}{\left(C_{k_{\widetilde\M},0}\right)^{2\alpha}} h^{-2}\int_{B_\M(0,h)}e^{-(1-\alpha)U(y)} h'F_j\varphi \int_{B_{m'}(0,1)} h^{-m} k_{\widetilde\M} \left(\frac{d_\M(x, y)}{h}, \abs{w}\right) w^j \:dwdy + O\left(h^{-2}(h')^2\right),
    \end{equation*}
    whose integrated term is zero by \cref{thm:1st 2nd 3rd moments k_h}. Thus
    \begin{multline}\label{eq:lifted laplacian URW alpha decomposition 1 final}
        h^{-2}\int_{\M}\int_{\widetilde\M_y} \frac{h^{-m} (h')^{-m'} k_{\widetilde\M} \left(\frac{d_\M(x, y)}{h}, \frac{d_{{\widetilde\M}_{x}}(\widetilde{\mathrm{P}}_{yx} u, q)}{h'}\right) \left(\varphi(y,q) - \varphi\left(y, \widetilde{\mathrm{P}}_{yx}u\right)\right)}{\left(\E\Big[\overline{k_{\widetilde\M}}(x,u)\Big]\right)^\alpha\left(\E\Big[\overline{k_{\widetilde\M}}(y,q)\Big]\right)^\alpha} \: d\Vol_{\widetilde\M_y}(q)d\mu(y) \\
        = O\left(h^{-2}(h')^2\right).
    \end{multline}
    \indent \textbf{Step 2 (Expansion along the base manifold)}. We expand the second term of the splitting in normal coordinates about both $x$ and $(y, \widetilde{\mathrm{P}}_{yx}u)$.
    As in Step 1, by replacing the denominator, we obtain
    \begin{align*}
        &h^{-2}\int_{\widetilde\M} \frac{h^{-m} (h')^{-m'} k_{\widetilde\M} \left(\frac{d_\M(x, y)}{h}, \frac{d_{{\widetilde\M}_{y}}(\widetilde{\mathrm{P}}_{yx} u, q)}{h'}\right) \left(\varphi\left(y, \widetilde{\mathrm{P}}_{yx}u\right) - \varphi(x,u)\right)}{\left(\E\Big[\overline{k_{\widetilde\M}}(x,u)\Big]\right)^\alpha\left(\E\Big[\overline{k_{\widetilde\M}}(y,q)\Big]\right)^\alpha}e^{-\widetilde U(y,q)} \: d\Vol_{\widetilde\M}(y,q) \\
            &=h^{-2}\frac{e^{(2\alpha-1)U(x)}}{\left(C_{k_{\widetilde\M},0}\right)^{2\alpha}} \int_{B_\M(0, h)}\int_{B_{T_{(y,\widetilde{\mathrm{P}}_{yx}u)}\widetilde\M_y}(0,1)} h^{-m}k_{\widetilde\M}\left(\frac{d_{\widetilde\M}(x,y)}{h}, \abs{w}\right) \left(\varphi\left(y, \widetilde{\mathrm{P}}_{yx}u\right) - \varphi(x,u)\right) \\
                &\quad
                \times e^{-(1-\alpha)\left(U(y) - U(x)\right)} \left(1 + O(h^2) + O((h')^2\right)\sqrt{\det \widetilde g^\V\left(\exp_{(y,\widetilde{\mathrm{P}}_{yx}u)}^{\widetilde\M_y}(h'w)\right)}\: dw d\Vol_\M(y). \numberthis\label{eq:Taylor lifted laplacian URW s decomposition 2 before coordinates}   
    \end{align*}
    Denote by $\gamma \in C^\infty(I;\M)$ the geodesic with $\gamma(0) = x$ and $\gamma'(0) = hv$ so that $\gamma(1) = y$. Its horizontal lift $\widetilde\gamma$ starting at $(x, u)$ is the unique horizontal geodesic satisfying
    \begin{equation*}
        \pi \circ \widetilde\gamma = \gamma, \quad \widetilde\gamma'(t) \in \H_{\widetilde\gamma}
    \end{equation*}
    with $\widetilde\gamma(0) = (x, u)$ and $\widetilde\gamma'(0) = h\widetilde v$ so that $\widetilde\gamma(1) = (y, \widetilde{\mathrm{P}}_{yx}u)$. Then
    \begin{equation*}
        \left(\varphi\left(y, \widetilde{\mathrm{P}}_{yx}u\right) - \varphi(x,u)\right) = \varphi \circ\widetilde\gamma(1) - \varphi \circ \widetilde\gamma(0) = \frac{d}{dt}\Big|_{t=0}\varphi \circ\widetilde\gamma(t) + \frac{1}{2}\frac{d^2}{dt^2}\Big|_{t=0}\varphi \circ\widetilde\gamma(t) + O(h^3),
    \end{equation*}
    where
    \begin{equation*}
        \frac{d}{dt}\varphi \circ\widetilde\gamma(t) = d_{\widetilde\gamma(t)}\varphi \circ \left(d\pi_{\widetilde\gamma(t)}\right)^{-1}\left(\gamma'(t)\right) = d_{\widetilde\gamma(t)}\varphi\left(\widetilde\gamma'(t)\right) = \widetilde\nabla\varphi \cdot \widetilde\gamma'(t) = \left(\widetilde\nabla\varphi\right)^\H \cdot \widetilde\gamma'(t),
    \end{equation*}
    and, since $\widetilde\gamma$ is an horizontal geodesic and the inner product is compatible with the connection $\widetilde\nabla$,
    \begin{equation*}
        \frac{d^2}{dt^2}\varphi \circ\widetilde\gamma(t) = \frac{d}{dt}\left(\left(\widetilde\nabla\varphi\right)^\H \cdot \widetilde\gamma'(t)\right) = \left(\widetilde\nabla_{\widetilde\gamma'(t)} \left(\widetilde\nabla\varphi\right)^\H\right) \cdot \widetilde\gamma'(t) + \left(\widetilde\nabla\varphi\right)^\H \cdot\widetilde\nabla_{\widetilde\gamma'(t)}\widetilde\gamma'(t) = \left(\widetilde\nabla_{\widetilde\gamma'(t)} \left(\widetilde\nabla\varphi\right)^\H\right) \cdot \widetilde\gamma'(t).
    \end{equation*}
    The expansion is thus given by
    \begin{equation}
        \left(\varphi\left(y, \widetilde{\mathrm{P}}_{yx}u\right) - \varphi(x,u)\right) = d_{(x,u)}\varphi\left(\widetilde v\right)h + \frac{1}{2}\Hess_x^\H\varphi\left(\widetilde v, \widetilde v\right)h^2 + O(h^3).
    \end{equation}
    We write $v = v^iE_i$ with orthonormal base $(E_i)_{1\le i\le m}$ of $T_x\M$ and horizontally lift it by linearity to $\widetilde v = v^i \widetilde E_i$ with the orthonormal basis $(\widetilde{E}_i)_{1\le i\le m}$ of $\H_{(x,u)}$; the above expansion then becomes
    \begin{align*}
        \left(\varphi\left(y, \widetilde{\mathrm{P}}_{yx}u\right) - \varphi(x,u)\right) &= \widetilde E_i \varphi v^i h + \frac{1}{2}\left(\left(\widetilde\nabla_{\widetilde E_i}\left(\widetilde\nabla\varphi\right)^\H\right) \cdot \widetilde E_j\right) v^iv^j h^2 + O(h^3) \\
            &= \widetilde E_i \varphi v^i h + \frac{1}{2}\left(\widetilde E_i \widetilde E_j - \left(\widetilde\nabla_{\widetilde E_i}E_j\right)^\H\right)\varphi v^iv^j h^2 + O(h^3). \numberthis\label{eq:expansion varphi decomposition 2 coordinates}
    \end{align*}
    Substituting \eqref{eq:expansion varphi decomposition 2 coordinates} into \eqref{eq:Taylor lifted laplacian URW s decomposition 2 before coordinates}, we obtain
    \begin{align*}
        &h^{-2}\frac{e^{(2\alpha-1)U(x)}}{\left(C_{k_{\widetilde\M},0}\right)^{2\alpha}} \int_{B_{m}(0,1)}\int_{B_{m'}(0,1)} k_{\widetilde\M}\left(\abs{v}, \abs{w}\right) \Bigg\{\left(\widetilde E_i \varphi v^i h + \frac{1}{2}\left(\widetilde E_i \widetilde E_j - \left(\widetilde\nabla_{\widetilde E_i}\widetilde E_j\right)^\H\right)\varphi v^iv^j h^2 + O(h^3)\right) \\
            &\quad \times \left(1 - (1 - \alpha)E_l U v^l h + O(h^2)\right) \left(1 + O(h^2) + O((h')^2\right)\left(1 + O((h')^2\right)\left(1 + O(h^2))\right)\: dw dv \\
        &=h^{-2}\frac{e^{(2\alpha-1)U(x)}}{\left(C_{k_{\widetilde\M},0}\right)^{2\alpha}} \int_{B_{m}(0, 1)}\int_{B_{m'}(0,1)} k_{\widetilde\M}\left(\abs{v}, \abs{w}\right) \\
            &\quad \times \Bigg\{\widetilde E_i \varphi v^i h + \frac{1}{2}\left(\widetilde E_i \widetilde E_j - \left(\widetilde\nabla_{\widetilde E_i}\widetilde E_j\right)^\H - 2(1-\alpha)E_iU\widetilde E_j\right)\varphi v^iv^j h^2 \Bigg\}\: dw dv + O(h). \numberthis\label{eq:Taylor lifted laplacian URW s decomposition 2 after coordinates}
    \end{align*}
    The first moment in \eqref{eq:Taylor lifted laplacian URW s decomposition 2 after coordinates} vanishes, and the second yields a weighted horizontal Laplacian by \cref{thm:1st 2nd 3rd moments k_h};
    \begin{align*}
        \frac{1}{2}\left(\widetilde E_i \widetilde E_j - \left(\widetilde\nabla_{\widetilde E_i}\widetilde E_j\right)^\H - 2(1-\alpha)E_iU\widetilde E_j\right)\varphi &\int_{B_{m}(0, 1)}\int_{B_{m'}(0,1)} v^iv^j \: dwdv \\
            &= \frac{C_{k_{\widetilde\M,2}}}{2}\left(\widetilde E_i \widetilde E_j - \left(\widetilde\nabla_{\widetilde E_i}\widetilde E_j\right)^\H - 2(1-\alpha)E_iU\widetilde E_j\right)\varphi\delta^{ij} \\
            &= \frac{C_{k_{\widetilde\M,2}}}{2}\sum_{i=1}^m \left(\widetilde E_i^2 - \left(\nabla_{\widetilde E_i}\widetilde E_i\right)^\H - 2(1-\alpha)\widetilde E_i \widetilde U \widetilde E_i\right)\varphi \\
            &\coloneqq \frac{C_{k_{\widetilde\M,2}}}{2} \Delta^\H_{2(1-\alpha)} \varphi.
    \end{align*}
    Thus,
    \begin{multline}\label{eq:lifted laplacian URW alpha decomposition 2 final}
        h^{-2}\int_{\M}\int_{\widetilde\M_y} \frac{h^{-m} (h')^{-m'} k_{\widetilde\M} \left(\frac{d_\M(x, y)}{h}, \frac{d_{{\widetilde\M}_{x}}(\widetilde{\mathrm{P}}_{yx} u, q)}{h'}\right) \left(\varphi\left(y, \widetilde{\mathrm{P}}_{yx}u\right) - \varphi(x,u)\right)}{\left(\E\Big[\overline{k_{\widetilde\M}}(x,u)\Big]\right)^\alpha\left(\E\Big[\overline{k_{\widetilde\M}}(y,q)\Big]\right)^\alpha} \: d\Vol_{\widetilde\M_y}(q)d\mu(y) \\
        = \frac{C_{k_{\widetilde\M},2}}{2\left(C_{k_{\widetilde\M},0}\right)^{2\alpha}}e^{(2\alpha-1)\widetilde U(x,u)}\Delta_{2(1-\alpha)}^\H \varphi(x,u) + O(h).
    \end{multline}
    We conclude by putting together \eqref{eq:lifted laplacian URW alpha decomposition 1 final} and \eqref{eq:lifted laplacian URW alpha decomposition 2 final}.
\end{proof}

{\bf Acknowledgment:} This publication is part of the project Interacting particle systems and Riemannian geometry (with project number OCENW.M20.251) of the research program Open Competitie ENW which is financed by the Dutch Research Council (NWO) \href{https://www.nwo.nl/en/projects/ocenwm20251}{\includegraphics[height=\fontcharht\font`\B]{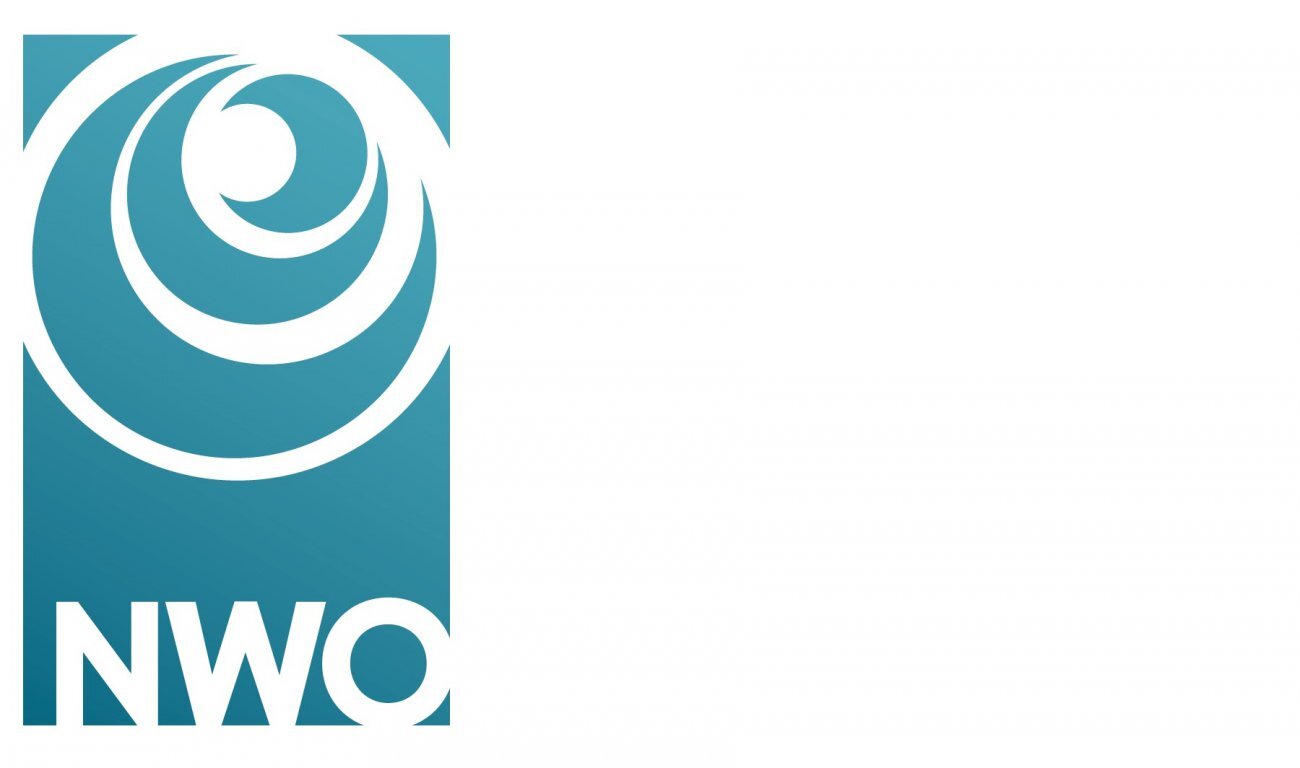}} \hspace{-10pt}.

\newpage
\printbibliography

\end{document}